\documentclass[10pt]{paper}
\usepackage[active]{srcltx}
\usepackage{amsmath}\usepackage{amssymb}\usepackage{amscd}\usepackage{amsthm}\usepackage{amsfonts}
\usepackage{empheq}
\usepackage{cite}
\usepackage{fullpage}
\usepackage[pagebackref=false,bookmarks=true,colorlinks=true,linktoc=page,citecolor=black,linkcolor=black,urlcolor=black]{hyperref}
\usepackage{tikz}
\usetikzlibrary{matrix,arrows,positioning,decorations.pathmorphing}

\newcommand{\R}{\mathbb{R}}
\newcommand{\Z}{\mathbb{Z}}
\newcommand{\eps}{\varepsilon}
\newcommand{\hH}{\widehat{H}}

\newcommand{\ti}{\textbf{i}}
\newcommand{\tj}{\textbf{j}}

\newcommand{\bl}{\boldsymbol{\lambda}}
\newcommand{\blambda}{\boldsymbol{\lambda}}
\newcommand{\some}{\diamond}

\newcommand{\idemp}{e(\ti)}
\newcommand{\idmep}{e(\ti)}
\newcommand{\idmepj}{e(\textbf{j})}

\newtheorem{theorem}{\textbf{Theorem}}[section]
\newtheorem{corollary}[theorem]{\textbf{Corollary}}
\newtheorem{lemma}[theorem]{\textbf{Lemma}}
\newtheorem{proposition}[theorem]{\textbf{Proposition}}
\newtheorem*{theo-intro}{\textbf{Theorem}}

\newtheorem*{claim*}{Claim}

\theoremstyle{definition}
\newtheorem{definition}[theorem]{\textbf{Definition}}

\theoremstyle{remark}
\newtheorem{remark}[theorem]{\textbf{Remark}}
\newtheorem{remarks}[theorem]{\textbf{Remarks}}

\numberwithin{equation}{section}

\begin{document} 

\title{Affine Hecke algebras of type D and generalisations of quiver Hecke algebras}
\author{L. Poulain d'Andecy and R. Walker}

\maketitle

\begin{abstract}
\let\thefootnote\relax\footnote{The authors thank Salim Rostam for interesting discussions and useful comments about this work. The first author is supported by \emph{Agence Nationale de la Recherche} through the JCJC project ANR-18-CE40-0001. The second author was supported in part by the European Research Council in the framework of The European Union H2020 with the Grant ERC 647353 QAffine.}We define and study cyclotomic quotients of affine Hecke algebras of type $D$. We establish an isomorphism between (direct sums of blocks of) these cyclotomic quotients and a generalisation of cyclotomic quiver Hecke algebras which are a family of $\Z$-graded algebras closely related to algebras introduced by Shan, Varagnolo and Vasserot. To achieve this, we first complete the study of cyclotomic quotients of affine Hecke algebras of type $B$ by considering the situation when a deformation parameter $p$ squares to 1. We then relate the two generalisations of quiver Hecke algebras showing that the one for type $D$ can be seen as fixed point subalgebras of their analogues for type $B$, and we carefully study how far this relation remains valid for cyclotomic quotients. This allows us to obtain the desired isomorphism. This isomorphism completes the family of isomorphisms relating affine Hecke algebras of classical types to (generalisations of) quiver Hecke algebras, originating in the famous result of Brundan and Kleshchev for the type A.
\end{abstract}

\section{Introduction}\label{sec-intro}

Khovanov-Lauda-Rouquier (KLR) algebras, often referred to as quiver Hecke algebras, are a family of graded algebras which have been introduced by Khovanov, Lauda \cite{KL} and Rouquier \cite{Rou} for the purposes of categorifying quantised enveloping algebras. Namely, the finite-dimensional graded projective modules over KLR algebras categorify the negative half of the quantised universal enveloping algebra of the Kac-Moody algebra associated to some Cartan datum.

There exist certain finite-dimensional quotients of KLR algebras, called cyclotomic KLR algebras, which categorify an irreducible highest weight module over a quantum Kac-Moody algebra. Brundan and Kleshchev \cite{BK} proved that cyclotomic KLR algebras (for type $A$ quivers) are isomorphic to blocks of cyclotomic Hecke algebras, which are certain finite-dimensional quotients of affine Hecke algebras of type $A$. Consequently there is a non-trivial $\mathbb{Z}$-grading on the cyclotomic Hecke algebras since cyclotomic KLR algebras are naturally $\Z$-graded. Furthermore, the deformation parameter does not appear in the defining relations of the cyclotomic KLR algebra. Rather, the deformation parameter appears only in the definition of the quiver which is fixed once the order of the deformation parameter is known.

\vskip .2cm
In our previous work \cite{PW} we defined and studied cyclotomic quotients of affine Hecke algebras of type $B$ by generalising the isomorphism of Brundan and Kleshchev. We first gave a family of isomorphisms between blocks of cyclotomic Hecke algebras and cyclotomic KLR algebras. We then extended one of these isomorphisms to construct an explicit isomorphism between cyclotomic quotients of affine Hecke algebras of type $B$ and quotients of a family of $\mathbb{Z}$-graded algebras defined by generators and relations depending upon an underlying quiver. In fact this family of algebras is a generalisation of the algebras introduced by Varagnolo--Vasserot in \cite{VV} where they considered representations of affine Hecke algebras of type $B$ such that certain elements do not have eigenvalues $\pm 1$. We did not impose this restriction in \cite{PW}. 

\vskip .2cm
The main purpose of this paper is to extend these results within the framework of affine Hecke algebras of type $D$. In order to do this, we use the well-known connection with the affine Hecke algebra of type $B$. The affine Hecke algebra of type $B$ has two deformation parameters, denoted $p$ and $q$, and when one deformation parameter ($p$ in our notation) squares to 1, it admits the affine Hecke algebra of type $D$ as a fixed point subalgebra with respect to a certain involutive algebra automorphism. By restricting this automorphism to cyclotomic quotients, we naturally obtain cyclotomic quotients for type $D$, extending the family of cyclotomic quotients of affine Hecke algebras of other types (namely, types $A$ and $B$). We note that, by their definition, understanding the representation theory of all these cyclotomic quotients allows one to understand all finite-dimensional representations of affine Hecke algebras of type $D$.

However, in \cite{PW}, we did not consider affine Hecke algebras of type $B$ with $p^2=1$. So our first task in this paper is to construct cyclotomic quotients of affine Hecke algebras of type $B$ when the deformation parameter $p$ satisfies $p^2=1$ so that, together with the work in \cite{PW}, we have considered all cyclotomic quotients of affine Hecke algebras of type $B$. We establish in this paper an isomorphism between these cyclotomic quotients (when $p^2=1$) and certain generalisations of cyclotomic quiver Hecke algebras for type $B$. We emphasise that one cannot simply take $p^2=1$ in the constructions of \cite{PW}. Indeed the resulting generalisation of quiver Hecke algebras of type B needed when $p^2=1$ turns out to be different than the one obtained when $p^2 \neq 1$. 

Fortunately, as one could expect, the situation is somewhat simpler when $p^2=1$ and we can use some of the results in \cite{PW} and provide the necessary modifications while avoiding a full repetition of the technical details in \cite {PW}. We note that, as in \cite{PW}, the main ideas behind these isomorphisms are inspired by the original construction of Brundan and Kleshchev for the type $A$ setting \cite{BK}. Furthermore, one could ask for a definition of a family of algebras placed between the two generalisations of quiver Hecke algebras of type B (when $p^2=1$ and when $p^2\neq 1$). This is the subject of another paper \cite{PR}.
    
\vskip .2cm
We then look for a suitable algebra to play the role of the cyclotomic quiver Hecke algebra for type $D$. This is found to be a family of $\Z$-graded algebras introduced by Shan, Varagnolo and Vasserot in \cite{SVV} which they used to prove a conjecture of Kashiwara and Miemietz \cite{KM} concerning the category of representations of affine Hecke algebras of type $D$. We note that in the work of both Kashiwara--Miemietz \cite{KM} and Shan--Varagnolo--Vasserot \cite{SVV}, a restriction was imposed on the representations of the affine Hecke algebras of type $D$. This restriction stated that $\pm1$ does not appear in the set of eigenvalues of certain elements (or equivalently, $\pm1$ does not appear in the vertex set of the quiver defining the algebras used by Shan, Varagnolo and Vasserot). Here we work with no restriction on the considered representations and thus we have to generalise the definition of the algebras of Shan, Varagnolo and Vasserot in order to include the possibility of $\pm1$ in the set of eigenvalues. We remark here that our convention of the roles of the deformation parameters $p$ and $q$ in this paper is the opposite convention to that used by Shan--Varagnolo--Vasserot in \cite{VV} and Walker in \cite{W}. 

Once the generalisation of quiver Hecke algebras for type $D$ is introduced, we prove our main result in two steps. First, we establish a connection with the quiver Hecke algebra for type $B$ when $p^2=1$. More precisely, we show that it is isomorphic to the fixed point subalgebra of the quiver Hecke algebra for type $B$ (for a certain involutive automorphism) when $p^2=1$. Then we are finally able to achieve our original objective, namely to show that cyclotomic quotients of affine Hecke algebras of type $D$ are isomorphic to our generalisation of cyclotomic quiver Hecke algebras for type $D$.

\vskip .2cm
The quiver Hecke algebras for type $D$ (and their cyclotomic versions) are naturally $\Z$-graded. Thus, the result in this paper provides a non-trivial $\Z$-grading on the cyclotomic quotients of affine Hecke algebras of type $D$. It suggests in particular that one can study a graded representation theory for the affine Hecke algebra of type $D$.

\vskip .1cm
Moreover, from the definitions, the quiver Hecke algebras for type $D$ depend only on the following information: a quiver and an involution on the set of vertices (the vertex sets are subsets of a field $K$ and the involution will be of the form $i\mapsto i^{-1}$). This has the following immediate consequence. If $q$ is a root of unity, let $e$ be the smallest non-zero positive integer such that $q^{2e}=1$; if $q$ is not a root of unity, set $e:=\infty$. As a direct consequence of our isomorphism theorem, cyclotomic quotients of affine Hecke algebras of type $D$ depend on $q$ only through $e$.

For the study of irreducible representations of the affine Hecke algebra of type $D$, it is noted in \cite{KM} that it is enough to consider certain cyclotomic quotients. Using our isomorphism, this can be formulated equivalently by saying that it is enough to consider quivers with vertex sets of the form $I_{\lambda}:=\{\lambda^{\epsilon}q^{2l}\ |\ \epsilon\in\{-1,1\}\,,\ \ l\in\Z\}$ for $\lambda\in K\setminus\{0\}$. However, for a complete study of the category of representations of $\hH(D_n)$, one needs a priori to consider arbitrary cyclotomic quotients, even if one can expect to reduce everything to the situation $l=1$, in the spirit of the Dipper--Mathas Morita equivalence for type $A$ \cite{DM}. This is the subject of\cite{PR}.

Assume in this paragraph that the vertex set of the quiver is of the form $I_{\lambda}:=\{\lambda^{\epsilon}q^{2l}\ |\ \epsilon\in\{-1,1\}\,,\ \ l\in\Z\}$ for $\lambda\in K\setminus\{0\}$.  Then one sees easily from the definition of the quiver Hecke algebra for type $D$ that it is enough to consider one of the following three situations for the choice of $\lambda$:
\[\mathbf{(a)}\ \lambda=1\ \qquad \mathbf{(b)}\ \lambda=q\ \qquad \mathbf{(c)}\ \lambda\notin\{\pm q^{\mathbb{Z}}\}\ .\]
Situation \textbf{(c)} corresponds to a quiver with two disconnected components exchanged by the involution $i\mapsto i^{-1}$. All $\lambda\notin\{\pm q^{\mathbb{Z}}\}$ then result in isomorphic algebras.

If $q$ is an odd root of unity then Case \textbf{(b)} reduces to Case \textbf{(a)}. If $q$ is not an odd root of unity then situation \textbf{(b)} corresponds to a quiver with a single connected component, stable under the involution and with no fixed point. Finally, Case \textbf{(a)} corresponds to a quiver with a single connected component which is stable under the involution and with at least one fixed point.

\paragraph{Notation.} We work with a field $K$ of characteristic different from 2, and we fix $p,q\in K\setminus\{0\}$ such that $p^2=1$.

\paragraph{Organisation of the paper.} In Section \ref{sec-def} we recall the definitions of the affine Hecke algebras $\hH(B_n)$,  $\hH(D_n)$. We also recall the automorphism of $\hH(B_n)$ which yields $\hH(D_n)$ as a fixed point subalgebra when $p^2=1$. In Section \ref{sec-defV} we give the definition of the algebras $V_{\bl}^{\beta}$, $W^\delta_{\bl}$ and an automorphism of $V_{\bl}^{\beta}$. We give a basis of the fixed point subalgebra of $V_{\bl}^{\beta}$ with respect to this automorphism and prove that this fixed point subalgebra is isomorphic to $W_{\bl}^{\beta}$. The definition of the cyclotomic quotients of $\hH(B_n)$ and of $V_{\bl}^n$ are given in Section \ref{sec-cyc} together with the proof that these quotients are isomorphic. In Section \ref{sec-cycD} we define the cyclotomic quotients of $\hH(D_n)$ and of $W_{\bl}^n$. We prove that idempotent subalgebras of cyclotomic quotients of $\hH(D_n)$ are isomorphic to cyclotomic quotients of $W_{\bl}^{\beta}$.

\section{Affine Hecke algebras of types B and D}\label{sec-def}

\subsection{Root datum}\label{subsection:root datum}

Let $n\geq 1$. Let $\{\eps_i\}_{i=1,\dots,n}$ be an orthonormal basis of the Euclidean space $\R^n$ with scalar product $(.,.)$ and let
\[L:=\bigoplus_{i=1}^n\Z\eps_i\ ,\ \ \ \ \alpha_0=2\eps_1\ \ \ \ \ \text{and}\ \ \ \ \ \alpha_i=\eps_{i+1}-\eps_i\,,\ i=1,\dots,n-1\ .\]
The set $\{\alpha_i\}_{i=0,\dots,n-1}$ is a set of simple roots for the root system $R=\{\pm2\eps_i,\pm\eps_i\pm\eps_j\}_{i,j=1,\dots,n}$ of type $B$. The lattice $L$ is the weight lattice of the root system $R$.

If $n\geq 2$, let
\[\alpha'_0=\eps_1+\eps_2\ \ \ \ \ \text{and}\ \ \ \ \ \alpha'_i=\eps_{i+1}-\eps_i\,,\ i=1,\dots,n-1\ .\]
The set $\{\alpha'_i\}_{i=0,\dots,n-1}$ is a set of simple roots for the root system $R'=\{\pm\eps_i\pm\eps_j,\,i\neq j\,\}_{i,j=1,\dots,n}$ of type $D$. 
\begin{remark}
The lattice $L$ is a lattice strictly included between the root lattice and the weight lattice of the root system $R'$.
\hfill$\triangle$
\end{remark}

For a root $\alpha\in R$, we identify its coroot $\alpha^{\vee}$ with the following element of $\text{Hom}_{\Z}(L,\Z)$:
\[\alpha^{\vee}\ :\ x\mapsto 2\frac{(\alpha,x)}{(\alpha,\alpha)}\ .\]
To each root $\alpha\in R$ is associated the element $r_{\alpha}\in\text{End}_{\Z}(L)$ defined by
\begin{equation}\label{act-W0}
r_{\alpha}(x):=x-\alpha^{\vee}(x)\,\alpha\ \ \ \ \ \ \text{for any $x\in L$.}
\end{equation}
The group generated by $r_{\alpha}$, $\alpha\in R$, is identified with the Weyl group $W(B_n)$ associated to $R$. Let $r_0,r_1,\dots,r_{n-1}$ be the elements of $W(B_n)$ corresponding to the simple roots $\alpha_0,\alpha_1,\dots,\alpha_{n-1}$.

The group generated by $r_{\alpha}$, $\alpha\in R'$, is identified with the Weyl group $W(D_n)$ associated to $R'$. Let $s_0,s_1,\dots,s_{n-1}$ be the elements of $W(D_n)$ corresponding to the simple roots $\alpha'_0,\alpha'_1,\dots,\alpha'_{n-1}$.

It is a classical fact that the following map
\[s_0\mapsto r_0r_1r_0\,,\ \ \ \ s_i\mapsto r_i\,,\ i=1,\dots,n-1,\]
gives an isomorphism between $W(D_n)$ and the subgroup of $W(B_n)$ generated by $r_0r_1r_0,\, r_1,\dots,r_{n-1}$. Equivalently, the subgroup $W(D_n)$ can also be described as the kernel of the character of $W(B_n)$ defined by $r_0\mapsto -1$ and $r_i\mapsto 1$ for $i=1,\dots,n-1$.

\subsection{Affine Hecke algebra of type $B$}

We set $q_0:=p$ and $q_i:=q$ for $i=1,\dots,n-1$ (only in this subsection do we not need the assumption that $p^2=1$). We denote by $\hH(B_n)$ the affine Hecke algebra corresponding to the root datum of type $B$ associated to the weight lattice $L$. The algebra $\hH(B_n)$ is the $K$-algebra generated by elements 
$$g_0,g_1,\dots,g_{n-1}\ \ \ \ \text{and}\ \ \ \ X^x,\ \ x\in L\ .$$
Defining relations are $X^0=1$, $X^xX^{x'}=X^{x+x'}$ for any $x,x'\in L$, and 
\begin{empheq}{alignat=3}
&g_i^2=(q_i-q_i^{-1})g_i+1 & \qquad & \text{for}\ i\in\{0,\dots,n-1\}\,,\label{rel-H1}\\
&g_0g_1g_0g_1=g_1g_0g_1g_0\,, && \label{rel-H2}\\
&g_ig_{i+1}g_i=g_{i+1}g_ig_{i+1} && \text{for $i\in\{1,\dots,n-2\}$}\,, \label{rel-H3}\\
&g_ig_j=g_jg_i && \text{for $i,j\in\{0,\dots,n-1\}$ such that $|i-j|>1$}\,.
\label{rel-H4}
\end{empheq}

together with
\begin{equation}\label{rel-Lu}
g_iX^x-X^{r_i(x)}g_i=(q_i-q_i^{-1})\frac{X^x-X^{r_i(x)}}{1-X^{-\alpha_i}}\ ,
\end{equation}
for any $x\in L$ and $i=0,1,\dots,n-1$. Note that the right hand side of (\ref{rel-Lu}) belongs to the subalgebra generated by elements $X^y$, $y\in L$, since there exists $k\in\mathbb{Z}$ such that $r_i(x)=x-k\alpha_i$.

Let $X_i:=X^{\eps_i}$ for $i=1,\dots,n$, so that $\{X^x\}_{x\in L}=\{X_1^{a_1}\dots X_n^{a_n}\}_{a_1,\dots,a_n\in\Z}$. Then the algebra $\hH(B_n)$ is generated by elements 
$$g_0,g_1,\dots,g_{n-1},X_1^{\pm1},\dots,X_n^{\pm1}\,,$$
and it is easy to check that a set of defining relations is (\ref{rel-H1})--(\ref{rel-H4}) together with
\begin{empheq}{alignat=3}
&X_iX_j=X_jX_i & \qquad &\text{for}\ i,j\in\{1,\dots,n\}\,,\label{rel-H5}\\
&g_0X_1^{-1}g_0=X_1\,, && \label{rel-H6}\\
&g_iX_ig_i=X_{i+1} && \text{for $i\in\{1,\dots,n-1\}$}\,, \label{rel-H7}\\
&g_iX_j=X_jg_i && \text{for $i\in\{0,\dots,n-1\}$ and $j\in\{1,\dots,n\}$ such that $j\neq i,i+1$}\,. \label{rel-H8}
\end{empheq}
For any element $w\in W(B_n)$, let $w=r_{a_1}\dots r_{a_k}$ be a reduced expression for $w$ in terms of the generators $r_0,\dots,r_{n-1}$, and define $g_w:=g_{a_1}\dots g_{a_k}\in \hH(B_n)$. This definition does not depend on the reduced expression for $w$. It is a standard fact that the following sets of elements are $K$-bases of $\hH(B_n)$:
\begin{equation}\label{base-aff}
\{X^xg_w\}_{x\in L,\,w\in W(B_n)}\ \ \ \ \text{and}\ \ \ \ \{g_wX^x\}_{x\in L,\,w\in W(B_n)} .
\end{equation}
\begin{remark}\label{braidB}
The Weyl group $W(B_n)$ is generated by $r_0,r_1,\dots,r_{n-1}$ and a set of defining relations is (\ref{rel-H1})--(\ref{rel-H4}) with $g_i$ replaced by $r_i$ and $q_i$ (that is, $p$ and $q$) replaced by 1. Relations (\ref{rel-H2})--(\ref{rel-H4}) are the braid relations of type $B$.
\hfill$\triangle$
\end{remark}

\subsection{Affine Hecke algebra of type $D$}

Let $n\geq 2$. We denote by $\hH(D_n)$ the affine Hecke algebra corresponding to the root datum of type $D$ associated to the lattice $L$. The algebra $\hH(D_n)$ is the $K$-algebra generated by elements 
$$T_0,T_1,\dots,T_{n-1}\ \ \ \ \text{and}\ \ \ \ X^x,\ \ x\in L\ .$$
Defining relations are $X^0=1$, $X^xX^{x'}=X^{x+x'}$ for any $x,x'\in L$, and
\begin{empheq}{alignat=3}
& T_i^2=(q-q^{-1})T_i+1  & \qquad & \text{for}\ i\in\{0,\dots,n-1\}\,,\label{rel-HD1}\\[0.2em]
& T_0T_2T_0=T_2T_0T_2\,, &&   \label{rel-HD2}\\[0.2em]
& T_0T_j=T_{j}T_0 &&  \text{for $j\in\{1,\dots,n-1\}\backslash\{2\}$}\,,\label{rel-HD3}\\[0.2em]
& T_iT_{i+1}T_i=T_{i+1}T_iT_{i+1} &&  \text{for $i\in\{1,\dots,n-2\}$}\,,\label{rel-HD4}\\[0.2em]
& T_iT_j=T_jT_i && \text{for $i,j\in\{1,\dots,n-1\}$ such that $|i-j|>1$}\,.\label{rel-HD5}
\end{empheq} 
together with
\begin{equation}\label{rel-LuD}
T_iX^x-X^{s_i(x)}T_i=(q-q^{-1})\frac{X^x-X^{s_i(x)}}{1-X^{-\alpha'_i}}\ ,
\end{equation}
for any $x\in L$ and $i=0,1,\dots,n-1$. Note that the right hand side of (\ref{rel-Lu}) belongs to the subalgebra generated by elements $X^y$, $y\in L$, since there exists $k\in\mathbb{Z}$ such that $s_i(x)=x-k\alpha'_i$.

Again, let $X_i:=X^{\eps_i}$ for $i=1,\dots,n$, so that $\{X^x\}_{x\in L}=\{X_1^{a_1}\dots X_n^{a_n}\}_{a_1,\dots,a_n\in\Z}$. Then the algebra $\hH(D_n)$ is generated by elements 
$$T_0,T_1,\dots,T_{n-1},X_1^{\pm1},\dots,X_n^{\pm1}\,,$$
and it is easy to check that a set of defining relations is (\ref{rel-HD1})--(\ref{rel-HD5}) together with
\begin{empheq}{alignat=3}
& X_iX_j=X_jX_i & \qquad & \text{for}\ i,j\in\{1,\dots,n\}\,,\\[0.2em]
& T_0X_1^{-1}T_0=X_2\,, && \\[0.2em]
& T_0X_j=X_jT_0 && \text{for $j\in\{3,\dots,n-1\}$}\,,\\[0.2em]
& T_iX_iT_i=X_{i+1} && \text{for $i\in\{1,\dots,n-1\}$}\,,\\[0.2em]
& T_iX_j=X_jT_i && \text{for $i\in\{1,\dots,n-1\}$ and $j\in\{1,\dots,n\}$ such that $j\neq i,i+1$}\,.
\end{empheq}
For any element $w\in W(D_n)$, let $w=s_{a_1}\dots s_{a_k}$ be a reduced expression for $w$ in terms of the generators $s_0,\dots,s_{n-1}$, and define $T_w:=T_{a_1}\dots T_{a_k}\in \hH(D_n)$. This definition does not depend on the reduced expression for $w$. It is a standard fact that the following sets of elements are $K$-bases of $\hH(D_n)$:
\begin{equation}\label{base-affD}
\{X^xT_w\}_{x\in L,\,w\in W(D_n)}\ \ \ \ \text{and}\ \ \ \ \{T_wX^x\}_{x\in L,\,w\in W(D_n)} .
\end{equation}
\begin{remark}\label{braidD}
The Weyl group $W(D_n)$ is generated by $s_0,s_1,\dots,s_{n-1}$ and a set of defining relations is (\ref{rel-HD1})--(\ref{rel-HD5}) with $T_i$ replaced by $s_i$ and $q$ replaced by 1. Relations (\ref{rel-HD2})--(\ref{rel-HD5}) are the braid relations of type D.
\hfill$\triangle$
\end{remark}

\paragraph{Fixed point subalgebra.} Recall that $p^2=1$. With this assumption, the following map
\[T_0\mapsto g_0g_1g_0\,,\ \ \ \ T_i\mapsto g_i\,,\ i=1,\dots,n-1,\ \ \ \ \ X^x\mapsto X^x\ ,\]
provides an isomorphism between $\hH(D_n)$ and the subalgebra of $\hH(B_n)$ generated by $g_0g_1g_0,g_1,\dots,g_{n-1}$ and $X^x$ (briefly, this standard statement can be checked by a direct verification of the homomorphism property and a comparison of the bases (\ref{base-aff}) and (\ref{base-affD})). The subalgebra of $\hH(B_n)$ isomorphic to $\hH(D_n)$ is the linear span of the following elements:
\[\{X^xg_w\}_{x\in L,\,w\in W(D_n)}\ ,\]
where we identified $W(D_n)$ as a subgroup of $W(B_n)$ by $s_0\mapsto r_0r_1r_0$ and $s_i\mapsto r_i$ for $i=1,\dots,n-1$.

Let $\eta$ be the involutive automorphism of $\hH(B_n)$ defined by
\begin{equation}\label{eta}
\eta\ :\ \ \ \ g_0\mapsto -g_0\,,\ \ \ \ g_i\mapsto g_i\,,\ i=1,\dots,n-1,\ \ \ \ \ \ X^x\mapsto X^x\ .
\end{equation}
We denote $\hH(B_n)^{\eta}:=\{x\in\hH(B_n)\ |\ \eta(x)=x\}$. Then it is straightforward to check that the subalgebra isomorphic to $\hH(D_n)$ inside $\hH(B_n)$ is $\hH(B_n)^{\eta}$ (this follows also immediately from the corresponding standard statement for finite Hecke algebras, which is a particular example of the description of cyclotomic Hecke algebras of $G(d,p,n)$ as fixed point subalgebras of the cyclotomic Hecke algebras of type $G(d,1,n)$ \cite{RR}; here we take $d=p=2$).

\section{Generalizations of quiver Hecke algebras}\label{sec-defV}

For $\lambda\in K\setminus\{0\}$, we set $I_{\lambda}:=\{\lambda^{\epsilon}q^{2l}\ |\ \epsilon\in\{-1,1\}\,,\ \ l\in\Z\}$ for $\lambda\in K\setminus\{0\}$. Fix an $l$-tuple $\bl=(\lambda_1,\dots,\lambda_l)$ of elements of $K\setminus\{0\}$ such that we have $I_{\lambda_a}\cap I_{\lambda_b}=\emptyset$ if $a\neq b$. Let
\[S_{\bl}:=I_{\lambda_1}\cup\dots\cup I_{\lambda_l}\ .\]
Let $\Gamma_{\bl}$ be the quiver whose vertices are indexed by $S_{\bl}$ (and we will identify a vertex with its index) and whose arrows are given as follows. For every $i\in S_{\bl}$ there is an arrow with origin $i$ and target $q^2i$.

\vskip .2cm
We will use the following notation, for $i,j\in S_{\bl}$:
\begin{equation*}
\begin{split}
i \nleftrightarrow j\quad &\textrm{ when } i\neq j \text{ and $j\notin\{q^2i,q^{-2}i\}$\,,} \\
i \rightarrow j \quad &\textrm{ when $j=q^2i\neq q^{-2}i$\,,}  \\
i \leftarrow j \quad &\textrm{ when $j=q^{-2}i\neq q^{2}i$\,,}  \\
i \leftrightarrow j \quad &\textrm{ when $j=q^{2}i=q^{-2}i$\,.}
\end{split}
\end{equation*}
Note that $i\neq j$ in all these cases (recall that $q^2\neq 1$) and that the fourth case means that $q^2=-1$ and $i=-j$.

We denote by $|i \rightarrow j|$ the number of arrows in $\Gamma_{\bl}$ which have origin $i$ and target $j$.
 
\subsection{The generalisation of quiver Hecke algebra for type $B$ with $p^2=1$}\label{sec-defV-psq=1}

The formulas
\begin{equation*}
\begin{split}
r_0\cdot(i_1,\ldots,i_n) &=(i_1^{-1},i_2,\ldots,i_n) \\
r_k\cdot(\ldots,i_k,i_{k+1},\ldots) &=(\ldots,i_{k+1},i_k,\ldots)\ \ \quad\text{for $k=1,\dots,n-1$,}
\end{split}
\end{equation*}
provide an action of the Weyl group $W(B_n)$ on $(S_{\bl})^n$. 
Let $\beta$ be a finite subset of $(S_{\bl})^n$ stable under this action of $W(B_n)$. In other words, let $\beta$ be a finite union of orbits in $(S_{\bl})^n$ for this action of $W(B_n)$.
\begin{definition} \label{definition:vv algebra}
The algebra $V_{\bl}^{\beta}$ is the unital $K$-algebra generated by elements 
\begin{equation*}
\{ \psi_a \}_{0\leq a\leq n-1} \cup \{ y_j \}_{1\leq j\leq n} \cup \{ \idemp \}_{\textbf{i} \in \beta}
\end{equation*}
which satisfy the following defining relations. 
\begin{align}
\sum_{\textbf{i} \in \beta} \idmep & = 1\,,\ \ \quad \idmep\idmepj = \delta_{\ti,\textbf{j}}\idmep\,,\ \ \ \forall\ti,\textbf{j}\in \beta\,, \label{Rel:V1} \\ 
y_iy_j&=y_jy_i\,,\ \ \quad y_{i}\idmep=\idmep y_{i}\,,\ \ \ \ \forall i,j\in\{1,\ldots,n\} \text{ and } \forall\ti\in\beta\,,\label{Rel:V2}
\end{align}
and, for $a,b\in\{0,1,\dots,n-1\}$ with $b\neq 0$, for $j\in\{1,\dots,n\}$, and for $\ti=(i_1,\dots,i_n)\in\beta$,
\begin{align}
\psi_{a}\idmep &= e(r_{a}(\ti))\psi_a\,,\label{Rel:V3}\\
(\psi_b y_j - y_{\pi_b(j)}\psi_b)\idmep&=\left\{
\begin{array}{l l l}
			-\idmep & \quad \textrm{ if } j=b,\ i_b=i_{b+1}\,,\\
			\idmep & \quad \textrm{ if } j=b+1,\ i_b=i_{b+1}\,,\\
			0 & \quad \textrm{ else,}
\end{array} \right. \label{Rel:V4}
\end{align}
where $\pi_b$ is the transposition of $b$ and $b+1$ acting on $\{1,\dots,n\}$;
\begin{align}
\psi_a\psi_b&=\psi_b\psi_a \hspace{0.3em} \textrm{ if } |a-b|>1\,, \label{Rel:V5} \\
\psi_b^2\idmep &= \left\{
\begin{array}{l l l l l}
0 & \quad \textrm{if } i_b=i_{b+1}\,, \\[0.2em]
\idmep & \quad \textrm{if } i_{b} \nleftrightarrow i_{b+1}\,, \\[0.2em]
(y_{b+1} - y_b)\idmep & \quad \textrm{if } i_b \rightarrow i_{b+1}\,,\\[0.2em]
(y_b - y_{b+1})\idmep & \quad \textrm{if } i_b \leftarrow i_{b+1}\,, \\[0.2em]
(y_{b+1}-y_b)(y_b-y_{b+1})\idemp & \quad \textrm{if } i_b \leftrightarrow i_{b+1}\,, 
\end{array} \right. \label{Rel:V6}\\
(\psi_{b}\psi_{b+1}\psi_{b} - \psi_{b+1}\psi_{b}\psi_{b+1})\idemp &= \left\{
\begin{array}{l l l l}
\idmep & \quad i_b=i_{b+2}\rightarrow i_{b+1}\,\\[0.2em]
-\idmep & \quad i_b=i_{b+2}\leftarrow i_{b+1}\, \\[0.2em]
(y_{b+2}-2y_{b+1}+y_b)\idmep & \quad i_b=i_{b+2} \leftrightarrow i_{b+1}\,,\\[0.2em]
0 & \quad \text{else\,,}\,
\end{array} \right. \label{Rel:V7}
\end{align}
\begin{align}
\psi_0 y_1 + y_1\psi_0 & =0\,,\label{Rel:V8}\\
\psi_0 y_j &= y_j \psi_0 \quad\textrm{ if } j > 1\,, \label{Rel:V9}\\
\psi_0^2&=1\,, \label{Rel:V10}\\
((\psi_0 \psi_1)^2 - (\psi_1\psi_0)^2)&=0\,.
\label{Rel:V11}
\end{align}
\end{definition}

It is straightforward to check that the algebra $V_{\bl}^{\beta}$ admits a $\Z$-grading given as follows,
\begin{equation*}
\begin{split}
&\textrm{deg}(\idmep)=0\,, \\ 
&\textrm{deg}(y_j)=2\,, \\
&\textrm{deg}(\psi_0)=0\,,\\
&\textrm{deg}(\psi_b \idmep)=\left\{
\begin{array}{l l l}	
|i_b \rightarrow i_{b+1}|+|i_b \leftarrow i_{b+1}| & \quad \textrm{ if  } i_b \neq i_{b+1}\,,\\
-2 & \quad \textrm{ if  } 	i_b = i_{b+1}.	
\end{array} \right. 
\end{split}
\end{equation*}

Let $\beta=\beta_1\cup\beta_2\dots\cup \beta_k$ be the union of orbits $\beta_1,\dots,\beta_k$ in $(S_{\bl})^n$ for the action of $W(B_n)$. For any $a\in\{1,\dots,k\}$, the element
\[e(\beta_a):=\sum_{\ti\in\beta_a}e(\ti)\ \]
is a central idempotent in $V_{\bl}^{\beta}$ using Relations (\ref{Rel:V1})--(\ref{Rel:V3}). Moreover, again using (\ref{Rel:V1}), we have that
\[e(\beta_a)e(\beta_b)=\delta_{a,b}e(\beta_a)\ \ \ \quad\text{and}\quad\ \ \ e(\beta_1)+\dots+e(\beta_k)=1\ .\]
Therefore the algebra $V_{\bl}^{\beta}$ decomposes as a direct sum of ideals:
\[V_{\bl}^{\beta}=e(\beta_1)V_{\bl}^{\beta}\oplus\dots\oplus e(\beta_k)V_{\bl}^{\beta}\ .\]
Moreover, for $a\in\{1,\dots,k\}$, the subalgebra $e(\beta_a)V_{\bl}^{\beta}$ is isomorphic to the quotient of $V_{\bl}^{\beta}$ by the relation $e(\beta_a)=1$. Multiplying by different idempotents $e(\ti)$, we see that this relation in $V_{\bl}^{\beta}$ is equivalent to the set of relations $e(\ti)=0$ for all $\ti\notin\beta_a$. A direct inspection of the defining relations (\ref{Rel:V1})--(\ref{Rel:V11}) shows finally that the subalgebra $e(\beta_a)V_{\bl}^{\beta}$ is isomorphic to $V_{\bl}^{\beta_a}$, and we conclude that:
\begin{equation}\label{dec-V}
V_{\bl}^{\beta}\cong V_{\bl}^{\beta_1}\oplus\dots\oplus V_{\bl}^{\beta_k}\ .
\end{equation}

\paragraph{A polynomial representation of $V^\beta_{\bl}$.}\label{sec-polyrep} There is an action of $W(B_n)$ by algebra automorphisms on a polynomial algebra $K[Y_1,\dots,Y_n]$ given by:
\begin{equation}\label{action}
\begin{split}
r_kY_l & =Y_{\pi_k(l)} \quad \textrm{for }1 \leq k < n, \\
r_0Y_l & = \left\lbrace
\begin{array}{ll}
-Y_l & \textrm{ if }l=1 \\
Y_l & \textrm{ if }l\neq 1
\end{array}
\right.
\end{split}
\end{equation}
where $\pi_k$ is the transposition of $k$ and $k+1$ acting on $\{1,\dots,n\}$.

We define the following direct sums of polynomial algebras: 
\begin{equation*}
\textrm{Pol}_\beta:=\bigoplus_{\ti \in \beta} \textrm{Pol}_{\ti}, \textrm{ where $\textrm{Pol}_{\ti}:= K[Y_1, \ldots, Y_n]$ for each $\ti\in\beta$},
\end{equation*}
and, for $\ti\in\beta$, we denote $e'(\ti)$ the idempotent of $\text{Pol}_{\beta}$ corresponding to the summand $\text{Pol}_{\ti}$ in the direct sum above (in other words, $e'(\ti)$ is 1 in the summand $\textrm{Pol}_{\ti}$ and $0$ in every other).
\\
Define a map from the set of generators of $V^\beta_{\bl}$ to $\text{End}_K(\text{Pol}_\beta)$ as follows
(where $f\in K[Y_1,\dots,Y_n]$ and $a\cdot -$ denotes multiplication on the left by $a$):
\begin{equation*}
\begin{split}
\idmep &\mapsto e'(\ti)\cdot - \\
y_k &\mapsto Y_k\cdot - \\
\psi_k &\mapsto \left( fe'(\ti) \mapsto 
\left\lbrace
\begin{array}{ll}
\displaystyle\frac{r_kf-f}{Y_k-Y_{k+1}}e'(\ti) & \textrm{ if } i_k=i_{k+1} \\ \\
(r_kf) e'\bigl(r_k(\ti)\bigr) & \textrm{ if } i_k\rightarrow i_{k+1} \textrm{ or }i_k \nleftrightarrow i_{k+1} \\ \\
(Y_{k+1}-Y_k)(r_kf)e'\bigl(r_k(\ti)\bigr) & \textrm{ if } i_k\leftarrow i_{k+1} \textrm{ or } i\leftrightarrow i_{k+1}
\end{array}
\right. \right) \\ \\
\psi_0 &\mapsto \Big( \ \  fe'(\ti)\ \mapsto\ (r_0f)e'\bigl(r_0(\ti)\bigr) \ \ \Big).
\end{split}
\end{equation*}
We claim that we can extend this map to an algebra morphism
\begin{equation*}
R: V^\beta_{\bl} \longrightarrow \textrm{End}_K(\textrm{Pol}_\beta)
\end{equation*}

We must check the defining relations of $V^\beta_{\bl}$ are preserved under $R$. Relations (3.1) - (3.7) have been verified previously (and are independent of whether or not some $i^2_k=1$), see for example \cite{VV}, Proposition 7.4 or \cite{KL}, Proposition 2.3. It remains to check Relations (\ref{Rel:V8}) - (\ref{Rel:V11}). 
\\
\\
Relations (\ref{Rel:V8}) - (\ref{Rel:V10}) are immediately seen to be preserved. For Relation (\ref{Rel:V11}) we define, for any $\ti \in \beta$,
\begin{equation*}
\begin{array}{lll}
h({\ti}):= \left\lbrace \begin{array}{ll}
-1 & \textrm{ if } i_1=i_{2} \\
|i_{2} \rightarrow i_1| & \textrm{ if }i_1\neq i_{2}, 
\end{array} \right.
\end{array}
\end{equation*}
We first note that $R(\psi_1)$ acts on $\textrm{Pol}_{\ti}$ by
\begin{equation*}
R(\psi_1)\bigl(fe'(\ti)\bigr)=(Y_2-Y_1)^{h(\ti)}(r_1-\delta_{i_1,i_2})e'\bigl(r_1(\ti)\bigr)\ ,
\end{equation*}
and therefore, we have that the action of $R(\psi_1)R(\psi_0)R(\psi_1)R(\psi_0)$ on $fe'(\ti)$ is (here and below we omit the composition symbol)
\begin{equation*}
\begin{split}
&(Y_2-Y_1)^{h\bigl(r_0r_1r_0(\ti)\bigr)}(r_1-\delta_{i_2^{-1},i_1^{-1}})(r_0)(Y_2-Y_1)^{h\bigl(r_0(\ti)\bigr)}(r_1-\delta_{i_1^{-1},i_2})(r_0)(f)e'\bigl(r_1r_0r_1r_0(\ti)\bigr) \\
=&(Y_2-Y_1)^{h\bigl(r_0r_1r_0(\ti)\bigr)}(Y_1+Y_2)^{h\bigl(s_0(\ti)\bigr)}A(f)e'\bigl(r_1r_0r_1r_0(\ti)\bigr)
\end{split}
\end{equation*}
where $A=(r_1-\delta_{i_2^{-1},i_1^{-1}})(r_0)(r_1-\delta_{i_1^{-1},i_2})(r_0)$. Similarly, the action of $R(\psi_0)R(\psi_1)R(\psi_0)R(\psi_1)$ on $fe'(\ti)$ is given by
\begin{equation*}
\begin{split}
&(r_0)(Y_2-Y_1)^{h\bigl(r_0r_1(\ti)\bigr)}(r_1-\delta_{i_2^{-1},i_1})(r_0)(Y_2-Y_1)^{h(\ti)}(r_1-\delta_{i_1,i_2})(f)e'\bigl(r_0r_1r_0r_1(\ti)\bigr) \\
=&(Y_1+Y_2)^{h\bigl(r_0r_1(\ti)\bigr)}(Y_2-Y_1)^{h(\ti)}B(f)e'\bigl(r_0r_1r_0r_1(\ti)\bigr)
\end{split}
\end{equation*}
where $B=(r_0)(r_1-\delta_{i_2^{-1},i_1})(r_0)(r_1-\delta_{i_1,i_2})$.
\\
\\
Next, we note that 
\begin{equation*}
h\bigl(r_0r_1r_0(\ti)\bigr)=h(\ti) \textrm{ and } h\bigl(r_0r_1(\ti)\bigr)=h\bigl(r_0(\ti)\bigr).
\end{equation*}
Finally, to complete the verification of this relation, we check that $A-B=0$ which is straightforward. Therefore we can extend $R$ to a morphism of algebras. In other words, we have a polynomial representation of $V^\beta_{\bl}$.

\paragraph{Bases of $V^\beta_{\bl}$.} For any $w\in W(B_n)$, fix a reduced expression $w=r_{i_1}\dots r_{i_l}$ in terms of the generators $r_0,r_1,\dots,r_{n-1}$ of  $W(B_n)$, and set $\psi_w:=\psi_{i_1}\dots\psi_{i_l}$. Note that $\psi_w$ depends on the chosen reduced expression for $w$.

\begin{theorem}[Basis Theorem]\label{basis-theorem1}
The set 
\begin{equation*}
\mathcal{B}:=\{ y_1^{m_1}\cdots y_n^{m_n}\psi_w\idmep \mid w \in W(B_n), (m_1,\ldots,m_n) \in (\mathbb{Z}_{\geq 0})^n, \textbf{i}\in \beta \}
\end{equation*}
is a $K$-basis for $V^\beta_{\bl}$.
\end{theorem}
After a direct inspection of the defining relations, an easy argument using the algebra anti-automorphism of $V^\beta_{\bl}$ sending each generator to itself shows that we could also put the $y_1^{m_1}\cdots y_n^{m_n}$ to the right of $\psi_w$.
\begin{proof}
We adapt some of the arguments used in \cite[Proposition 7.5]{VV}. Due to the conciseness of the proof there, we find it useful to give more details here.

From the defining relations of $V^\beta_{\bl}$, for any product of the generators we can always move the $\psi_j$'s past the $e(\ti)$'s and $y_k$'s. So let us consider a word $\psi_{i_1}\dots\psi_{i_l}$. If $r_{i_1}\dots r_{i_l}$ is a reduced expression, we can use the braid relations (see Section \ref{sec-def}) to transform $r_{i_1}\dots r_{i_l}$ into the reduced expression chosen for $w=r_{i_1}\dots r_{i_l}\in W(B_n)$. In the algebra $V^\beta_{\bl}$, this results in $\psi_{i_1}\dots\psi_{i_l}=\psi_w+...$, where the correcting terms are $K[y_1,\dots,y_n]$-linear combinations of words with strictly less than $l$ letters in $\psi_0,\dots,\psi_{n-1}$ multiplied by idempotents $e(\ti)$. Then we use induction on $l$ to treat these words.

If $r_{i_1}\dots r_{i_l}$ is not reduced, we can use the braid relations to make an $r_i^2$ appear in the word. In the algebra $V^\beta_{\bl}$, this results in $\psi_{i_1}\dots\psi_{i_l}=\psi_{i'_1}\dots\psi_i^2\dots\psi_{i'_{l}}+...$, where the correcting terms are as in the preceding case. Using the defining relations for $\psi_i^2$, we again find ourselves in a situation where we can use the induction hypothesis on $l$.

\vskip .2cm
So the set $\mathcal{B}$ is a spanning set for $V^\beta_{\bl}$ and it remains to show that the elements in $\mathcal{B}$ are linearly independent. 

First we note that, for $w\in W(B_n)$, the action $f\mapsto wf$ on $K[Y_1,\dots,Y_n]$ extends to a field automorphism of $K(Y_1,\dots,Y_n)$. So by Dedekind's Lemma, the set of maps $\{f\mapsto wf\}_{w\in W(B_n)}$ are linearly independent over $K(Y_1,\dots,Y_n)$.

Second, we note that we have, for $f\in K[Y_1,\dots,Y_n]$ and $\ti\in\beta$,
\begin{equation}\label{eqpolrep}
R(\psi_w)\bigl(fe'(\ti)\bigr)=\Bigl(P_{w,\ti}\cdot wf+\sum_{w'<w}Q_{w',w,\ti}\cdot w'f\Bigr)e'\bigl(w(\ti)\bigr)\,,
\end{equation}
where $<$ denotes the Bruhat order of $W(B_n)$ and $P_{w,\ti},Q_{w',w,\ti}\in K(Y_1,\dots,Y_n)$ with $P_{w,\ti}\neq 0$ (this is easily checked by induction on the length of $w$).

Now, we have a decomposition of $V^\beta_{\bl}$ as a direct sum of its subspaces $e(\tj)V^\beta_{\bl}e(\ti)$. From $e(\tj)\psi_w=\psi_we\bigl(w^{-1}(\tj)\bigr)$, we note that $e(\tj)V^\beta_{\bl}e(\ti)$ is spanned by the elements $y_1^{m_1}\cdots y_n^{m_n}\psi_we(\ti)$ with $w$ such that $w(\ti)=\tj$. So we have to prove that these elements are linearly independent.

We fix $\ti,\tj\in\beta$ and we assume that we have
\[\sum_{w,\textbf{m}}c_{w,\textbf{m}}y_1^{m_1}\cdots y_n^{m_n}\psi_we(\ti)=0\ \ \ \ \ \ \ \text{($\textbf{m}$ denotes $(m_1,\dots,m_n)$)\,,}\]
for some $c_{w,\textbf{m}}\in K$, and where the sum is taken over $\textbf{m}\in(\Z_{\geq0})^n$ and $w\in W(B_n)$ such that $w(\ti)=\tj$. Applying the map $R$ and using (\ref{eqpolrep}), we obtain for any $f\in K[Y_1,\dots,Y_n]$
\[\sum_{w,\textbf{m}}c_{w,\textbf{m}}Y_1^{m_1}\cdots Y_n^{m_n}\Bigl(P_{w,\ti}\cdot wf+\sum_{w'<w}Q_{w',w,\ti}\cdot w'f\Bigr)e'(\tj)=0\ .\]
Then let $M\subset W(B_n)$ be the set of maximal elements for the Bruhat order such that $c_{w,\textbf{m}}\neq0$ for some $\textbf{m}$. From the $K(Y_1,\dots,Y_n)$-linear independence of $\{f\mapsto wf\}_{w\in W(B_n)}$, we deduce that
\[\sum_{\textbf{m}}c_{w,\textbf{m}}Y_1^{m_1}\cdots Y_n^{m_n}P_{w,\ti}=0\ \ \ \ \ \forall w\in M\,,\]
and since $P_{w,\ti}\neq 0$, we conclude that $c_{w,\textbf{m}}=0$ for any $\textbf{m}$. This shows that $M$ is empty and in turn that $c_{w,\textbf{m}}=0$ for any $w$ and $\textbf{m}$.
\end{proof}

\subsection{The generalisation of quiver Hecke algebra for type $D$}\label{sec-defW}

Let $n\geq 2$. The formulas
\begin{equation*}
\begin{split}
s_0\cdot(i_1,\ldots,i_n) &=(i_2^{-1},i_1^{-1},i_3,\ldots,i_n) \\
s_k\cdot(\ldots,i_k,i_{k+1},\ldots) &=(\ldots,i_{k+1},i_k,\ldots)\ \ \quad\text{for $k=1,\dots,n-1$,}
\end{split}
\end{equation*}
provide an action of the Weyl group $W(D_n)$ on $(S_{\bl})^n$. 
Let $\delta$ be a finite subset of $(S_{\bl})^n$ stable under this action of $W(D_n)$ or, in other words, let $\delta$ be a finite union of $W(D_n)$-orbits in $(S_{\bl})^n$.
\begin{definition} \label{definition:vv algebraD}
The algebra $W_{\bl}^{\delta}$ is the unital $K$-algebra generated by elements 
\begin{equation*}
\{ \Psi_a \}_{0\leq a\leq n-1} \cup \{ y_j \}_{1\leq j\leq n} \cup \{ \idemp \}_{\textbf{i} \in \delta}
\end{equation*}
which satisfy the following defining relations. 
\begin{align}
\sum_{\textbf{i} \in \delta} \idmep & = 1\,,\ \ \quad \idmep\idmepj = \delta_{\ti,\textbf{j}}\idmep\,,\ \ \ \forall\ti,\textbf{j}\in \delta\,, \label{Rel:V1D} \\ 
y_iy_j&=y_jy_i\,,\ \ \quad y_{i}\idmep=\idmep y_{i}\,,\ \ \ \ \forall i,j\in\{1,\ldots,n\} \text{ and } \forall\ti\in\delta\,,\label{Rel:V2D}
\end{align}
and, for $a,b,b'\in\{0,1,\dots,n-1\}$ with $b,b'\neq 0$, for $j\in\{1,\dots,n\}$, and for $\ti=(i_1,\dots,i_n)\in\delta$,
\begin{align}
\Psi_{a}\idmep &= e(s_{a}(\ti))\Psi_a\,,\label{Rel:V3D}\\
(\Psi_b y_j - y_{\pi_b(j)}\Psi_b)\idmep&=\left\{
\begin{array}{l l l}
			-\idmep & \quad \textrm{ if } j=b,\ i_b=i_{b+1}\,,\\
			\idmep & \quad \textrm{ if } j=b+1,\ i_b=i_{b+1}\,,\\
			0 & \quad \textrm{ else,}
\end{array} \right. \label{Rel:V4D}
\end{align}
where $\pi_b$ is the transposition of $b$ and $b+1$ acting on $\{1,\dots,n\}$;
\begin{align}
\Psi_b\Psi_{b'}&=\Psi_{b'}\Psi_b \hspace{0.3em} \textrm{ if } |b'-b|>1\,, \label{Rel:V5D} \\
\Psi_b^2\idmep &= \left\{
\begin{array}{l l l l l}
0 & \quad \textrm{if } i_b=i_{b+1}\,, \\[0.2em]
\idmep & \quad \textrm{if } i_{b} \nleftrightarrow i_{b+1}\,, \\[0.2em]
(y_{b+1} - y_b)\idmep & \quad \textrm{if } i_b \rightarrow i_{b+1}\,,\\[0.2em]
(y_b - y_{b+1})\idmep & \quad \textrm{if } i_b \leftarrow i_{b+1}\,, \\[0.2em]
(y_{b+1}-y_b)(y_b-y_{b+1})\idemp & \quad \textrm{if } i_b \leftrightarrow i_{b+1}\,, 
\end{array} \right. \label{Rel:V6D}\\
(\Psi_{b}\Psi_{b+1}\Psi_{b} - \Psi_{b+1}\Psi_{b}\Psi_{b+1})\idemp &= \left\{
\begin{array}{l l l l}
\idmep & \quad i_b=i_{b+2}\rightarrow i_{b+1}\,, \\[0.2em]
-\idmep & \quad i_b=i_{b+2}\leftarrow i_{b+1}\,, \\[0.2em]
(y_{b+2}-2y_{b+1}+y_b)\idmep & \quad i_b=i_{b+2} \leftrightarrow i_{b+1}\,,\\[0.2em]
0 & \quad \text{else\,,}\,
\end{array} \right. \label{Rel:V7D}
\end{align}
\begin{align}
\Psi_0 y_j\idmep & =\left\{
\begin{array}{l l}
(-y_{\pi_1(j)}\Psi_0+1)\idmep & \quad j\in\{1,2\}\ \text{and}\ i_1^{-1}=i_2\,,\\[0.2em]
-y_{\pi_1(j)}\Psi_0\idmep & \quad j\in\{1,2\}\ \text{and}\ i_1^{-1}\neq i_2\,,\\[0.2em]
y_j\Psi_0\idmep & \quad \text{else\,,}\,
\end{array} \right.\label{Rel:V8D}
\end{align}
where $\pi_1$ is the transposition of $1$ and $2$;
\begin{align}
\Psi_0 \Psi_b &= \Psi_b \Psi_0 \quad\textrm{ if } b=1\ \text{or}\ b> 2\,, \label{Rel:V9D}\\
\Psi_0^2\idmep &=\left\{
\begin{array}{l l l l l}
0 & \quad \textrm{if } i_1^{-1}=i_{2}\,, \\[0.2em]
\idmep & \quad \textrm{if } i_{1}^{-1} \nleftrightarrow i_{2}\,, \\[0.2em]
(y_1+y_2)\idmep & \quad \textrm{if } i_1^{-1} \rightarrow i_{2}\,,\\[0.2em]
-(y_1+ y_{2})\idmep & \quad \textrm{if } i_1^{-1} \leftarrow i_{2}\,, \\[0.2em]
-(y_1+ y_{2})^2\idemp & \quad \textrm{if } i_1^{-1} \leftrightarrow i_{2}\,, 
\end{array} \right. \label{Rel:V10D}\\
(\Psi_0 \Psi_2\Psi_0 - \Psi_2\Psi_0\Psi_2)\idmep &= \left\{
\begin{array}{l l l l}
\idmep & \quad i_1^{-1}=i_{3}\rightarrow i_{2}\,, \\[0.2em]
-\idmep & \quad i_1^{-1}=i_{3}\leftarrow i_{2}\,, \\[0.2em]
(y_{3}-2y_{2}-y_1)\idmep & \quad i_1^{-1}=i_{3} \leftrightarrow i_{2}\,,\\[0.2em]
0 & \quad \text{else\,.}\,
\end{array} \right.
\label{Rel:V11D}
\end{align}
\end{definition}

It is straightforward to check that the algebra $W_{\bl}^{\delta}$ admits a $\Z$-grading given by
\begin{equation*}
\begin{split}
&\textrm{deg}(\idmep)=0\,, \\ 
&\textrm{deg}(y_j)=2\,, \\
&\textrm{deg}(\Psi_0\idmep)=\left\{
\begin{array}{l l l}	
|i_1^{-1} \rightarrow i_{2}|+|i_1^{-1} \leftarrow i_{2}| & \quad \textrm{ if  } i_1^{-1} \neq i_{2}\,,\\
-2 & \quad \textrm{ if  } 	i_1^{-1} = i_{2}\,,	
\end{array} \right. \\
&\textrm{deg}(\Psi_b \idmep)=\left\{
\begin{array}{l l l}	
|i_b \rightarrow i_{b+1}|+|i_b \leftarrow i_{b+1}| & \quad \textrm{ if  } i_b \neq i_{b+1}\,,\\
-2 & \quad \textrm{ if  } 	i_b = i_{b+1}.	
\end{array} \right. 
\end{split}
\end{equation*}

If $\delta=\delta_1\cup\delta_2\dots\cup \delta_k$ is the union of orbits $\delta_1,\dots,\delta_k$ in $(S_{\bl})^n$ for the action of $W(D_n)$ we have, exactly as we obtained (\ref{dec-V}) above, that
\begin{equation}\label{dec-VD}
W_{\bl}^{\delta}\cong W_{\bl}^{\delta_1}\oplus\dots\oplus W_{\bl}^{\delta_k}\ .
\end{equation}

\begin{remarks}\label{rem-orb}
$\bullet$ Let $\beta$ be an orbit in $(S_{\bl})^n$ for the action of $W(B_n)$. Then $\beta$ is invariant under the action of $W(D_n)$ and decomposes into one or two $W(D_n)$-orbits. Assume that $\beta$ is the orbit of an element $(i_1,\dots,i_n)\in (S_{\bl})^n$. Then we have
\[\beta=\{\ (i_{\pi(1)}^{\epsilon_1},\dots,i_{\pi(n)}^{\epsilon_n})\ |\ \pi\in\mathfrak{S}_n\,, \epsilon_1,\dots,\epsilon_n\in\{\pm1\}\ \}\,,\]
where $\mathfrak{S}_n$ is the symmetric group on $\{1,\dots,n\}$. Now set
\[\beta_+:=\{\,(i_{\pi(1)}^{\epsilon_1},\dots,i_{\pi(n)}^{\epsilon_n})\ |\ \pi\in\mathfrak{S}_n\,, \epsilon_1,\dots,\epsilon_n\in\{\pm1\}\ \text{with an even number of $-1$}\,\}\ \ \ \ \ \text{and}\ \ \ \ \beta_-:=\beta\backslash\beta_+\ .\]
Then $\beta_+$ is the orbit of $(i_1,\dots,i_n)$ under the action of $W(D_n)$. We have two possibilities:

$-$ If $i_a^2\neq 1$ for any $a=1,\dots,n$, then $\beta=\beta_+\sqcup\beta_-$ decomposes as two different $W(D_n)$-orbits of the same size $|\beta|/2$.

$-$ If there is $a\in\{1,\dots,n\}$ such that $i_a^2=1$, then $\beta=\beta_+$ is a single $W(D_n)$-orbit (and $\beta_-$ is empty).

$\bullet$ We recall that there is a bijection between the orbits of $W(B_n)$ in $(S_{\bl})^n$ and a certain subset of dimension vectors for the quiver $\Gamma_{\bl}$ of height $n$ \cite[Remark 2.5]{PW}. More precisely, a dimension vector for $\Gamma_{\bl}$ is a map $\nu$ from $S_{\bl}$ to $\Z_{\geq0}$ with finite support, and the orbits of $W(B_n)$ in $(S_{\bl})^n$ are in bijection with the set
\begin{equation*}
\left\{\,\nu\,:\,S_{\bl}\to \Z_{\geq0}\ \bigr|\ \nu(i^{-1})= \nu(i)\ \ \forall i\in S_{\bl}\,,\ \ \ \textrm{and}\ \ \ \frac{1}{2}\sum_{\substack{i\in S_{\bl}\\i^2\neq 1}} \nu(i)+\sum_{\substack{i\in S_{\bl}\\i^2=1}} \nu(i)=n \right\}\,.
\end{equation*}
A bijection is given by associating to an orbit $\beta$ of $W(B_n)$ in $(S_{\bl})^n$ the map $\nu$ such that, for any $i\in S_{\bl}$, the number $\nu(i)$ is equal to the number of occurrences of $i$ and $i^{-1}$ in any element of $\beta$. Therefore, the $W(D_n)$-orbits which are also $W(B_n)$-orbits correspond to dimension vectors $\nu$ (as above) which are non-zero on at least one element $i$ such that $i^2=1$ while pairs of different $W(D_n)$-orbits, the union of which form a single $W(B_n)$-orbit, correspond to dimension vectors being zero on the set of elements $i$ such that $i^2=1$.
\hfill$\triangle$
\end{remarks}

\subsection{The algebra $W_{\bl}^{\beta}$ and a subalgebra of fixed points of $V_{\bl}^{\beta}$}

Let $n\geq 2$ and let $\beta$ be a finite union of $W(B_n)$-orbits in $(S_{\bl})^n$. The set $\beta$ is also a finite union of $W(D_n)$-orbits in $(S_{\bl})^n$, so that both algebras $V_{\bl}^{\beta}$ and $W_{\bl}^{\beta}$ are defined. From the definition of the algebra $V_{\bl}^{\beta}$ it is immediate that the following map gives an involutive automorphism of the algebra:
\[\rho\ :\ \ \ \psi_0\mapsto-\psi_0\,,\ \ \ \psi_b\mapsto\psi_b\,,\ b=1,\dots,n-1,\ \ \ \ y_j\mapsto y_j\,,\ j=1,\dots,n,\ \ \ \ \ e(\ti)\mapsto e(\ti)\,,\ \ti\in\beta\ .\]
We denote by $\bigl(V_{\bl}^{\beta}\bigr)^{\rho}:=\{x\in V_{\bl}^{\beta}\ |\ \rho(x)=x\}$ the subalgebra of $V_{\bl}^{\beta}$ consisting of elements fixed by $\rho$.

\paragraph{A choice of basis for $V_{\bl}^{\beta}$.} Let $\iota:W(D_n)\longrightarrow W(B_n)$ denote the standard inclusion map from Subsection \ref{subsection:root datum} defined by: $s_0\mapsto r_0r_1r_0$ and $s_k \mapsto r_k$ for $k\neq 0$.  The following lemma will allow us to choose reduced expressions compatible with $\iota$.
\begin{lemma}\label{lemBD}
Let $w\in W(B_n)$ such that $w\in\iota\bigl(W(D_n)\bigr)$. There is a reduced expression of $w$ in terms of $r_0,r_1,\dots,r_{n-1}$ which is also of the form $\iota(s_{i_1})\dots \iota(s_{i_k})$ for a reduced expression $s_{i_1}\dots s_{i_k}$ in $W(D_n)$ in terms of $s_0,s_1,\dots,s_{n_1}$.
\end{lemma}
\begin{proof}
$\bullet$ We start with the fact that any element in the subgroup of $W(B_n)$ generated by $r_1,\dots,r_{n-1}$ (isomorphic to the symmetric group on $n$ letters) has a reduced expression with at most one $r_1$. 

So let $r_{a_1}\dots r_{a_l}\in W(B_n)$ be a reduced expression with $a_1,\dots,a_l\in\{1,\dots,n-1\}$. Assume that it contains more than one occurrence of $r_1$. Then choose two of these occurrences with no other $r_1$ between them. By induction on $n$ (the subgroup generated by $r_2,\dots,r_{n-1}$ is isomorphic to the symmetric group on $n-1$ letters), we can assume that there is at most one occurrence of $r_2$ between them. If there is no such $r_2$ then we can use the braid relations (see Remark \ref{braidB}) to put the two $r_1$ next to each other, which contradicts the fact that we started with a reduced expression. If there is one $r_2$ then we can use the braid relations to create the subword $r_1r_2r_1$ and transform it into $r_2r_1r_2$. We thus reduced the number of occurrences of $r_1$.

$\bullet$ Now let $u=r_0r_{a_1}\dots r_{a_l}r_0\in W(B_n)$ be a reduced expression with $a_1,\dots,a_l\in\{1,\dots,n-1\}$. We just explained that there is a reduced expression of $r_{a_1}\dots r_{a_l}$ with at most one occurrence of $r_1$. In fact here $r_1$ must appear once otherwise we could use the braid relations to put the two $r_0$ next to each other and thus the expression would not be reduced. So we conclude that there is a reduced expression of the form $u=xr_0r_1r_0y$ with $x,y$ only in terms of $r_2,\dots,r_{n-1}$.

Finally, since $w\in\iota\bigl(W(D_n)\bigr)$ there is an even number of occurrences of $r_0$ in any of its reduced expressions, and therefore the above discussion implies the existence of a reduced expression for $w$ which is of the form $\iota(s_{i_1})\dots \iota(s_{i_k})$.

$\bullet$ If $s_{i_1}\dots s_{i_k}\in W(D_n)$ were not reduced we could remove (by cancellation law) two $s_i$ in this expression and this would give, by application of $\iota$, an expression for $w$ strictly smaller than its reduced expression $\iota(s_{i_1})\dots \iota(s_{i_k})$, which is impossible.
\end{proof}

Using the lemma, we now fix for $w\in W(B_n)$ a reduced expression in terms of $r_0,r_1,\dots,r_{n-1}$ such that if $w\in\iota\bigl(W(D_n)\bigr)$ then the reduced expression is of the form $\iota(s_{i_1})\dots \iota(s_{i_k})$ for a reduced expression $s_{i_1}\dots s_{i_k}$ in $W(D_n)$ (if $w\notin\iota\bigl(W(D_n)\bigr)$ then any reduced expression is allowed).

And for $w\in W(D_n)$, we fix the reduced expression $s_{i_1}\dots s_{i_k}$ in terms of $s_0,s_1,\dots,s_{n-1}$ such that $\iota(s_{i_1})\dots \iota(s_{i_k})$ is the chosen reduced expression for $\iota(w)$.

\begin{remark}
As an example, note that $r_0r_1r_2r_1r_0$ is a reduced expression that we are not allowed to choose. Using the braid relations, we find that an allowed one is $r_2r_0r_1r_0r_2=\iota(s_2)\iota(s_0)\iota(s_2)$. Note moreover that $\psi_0\psi_1\psi_2\psi_1\psi_0$ is not equal to $\psi_2\psi_0\psi_1\psi_0\psi_2$.

Besides, note also that $\iota(s_{i_1})\dots \iota(s_{i_k})$ does not have to be reduced in $W(B_n)$ even if $s_{i_1}\dots s_{i_k}$ is reduced in $W(D_n)$. For example, $s_0s_2s_0$ is reduced while $\iota(s_0)\iota(s_2)\iota(s_0)=r_0r_1r_0r_2r_0r_1r_0$ is not.\hfill$\triangle$
\end{remark}

\vskip .2cm
Once these choices of reduced expressions are made, we define, as in Subsection \ref{sec-defV-psq=1}, $\psi_w:=\psi_{i_1}\dots \psi_{i_l}\in V_{\bl}^{\beta}$ where $r_{i_1}\dots r_{i_l}$ is the chosen reduced expression for $w\in W(B_n)$. Similarly, for $w\in W(D_n)$ we set $\Psi_w:=\Psi_{j_1}\dots \Psi_{j_k}\in W_{\bl}^{\beta}$ where $s_{j_1}\dots s_{j_k}$ is the chosen reduced expression for $w$.

Note that $r_0r_1r_0$ is the unique reduced expression for $\iota(s_0)$ and therefore $\psi_{\iota(s_0)}=\psi_0\psi_1\psi_0$, while $\psi_{\iota(s_k)}=\psi_k$ for $k=1,\dots,n-1$.
\begin{theorem}[Basis Theorem 2]\label{prop:fixed point VV is SVV}
The set 
\begin{equation*}
\mathcal{D}:=\{ y_1^{m_1}\cdots y_n^{m_n}\Psi_w\idmep \mid w \in W(D_n), (m_1,\ldots,m_n) \in (\mathbb{Z}_{\geq 0})^n, \textbf{i}\in \beta \}
\end{equation*}
is a $K$-basis for $W^\beta_{\bl}$ and $W^\beta_{\bl}$ is isomorphic to the subalgebra of fixed points $\bigl(V_{\bl}^{\beta}\bigr)^{\rho}$ of $V_{\bl}^{\beta}$.
\end{theorem}
\begin{proof}
First, we can reproduce the same argument as in the beginning of the proof of Theorem \ref{basis-theorem1} to show that the set
$\mathcal{D}$ is a spanning set for $W^\beta_{\bl}$. So we do not repeat the details.

Then define a map from the set of generators of $W^\beta_{\bl}$ to $V_{\bl}^{\beta}$:
\[\Psi_k\mapsto \psi_{\iota(s_k)}\ (k=0,1,\dots,n-1)\,,\ \ \ y_j\mapsto y_j\ (j=1,\dots,n)\,,\ \ \ \idmep\mapsto\idmep\ (\ti\in\beta)\ .\]
We claim that this map extends to an algebra homomorphism:
\begin{equation}\label{iso-fixed-points}
\phi\ :\ W^\beta_{\bl} \ \longrightarrow\ \bigl(V_{\bl}^{\beta}\bigr)^{\rho}\ .
\end{equation}
First, it is obvious that the map indeed takes values in $\bigl(V_{\bl}^{\beta}\bigr)^{\rho}$. So it remains to check that the defining relations of $W^\beta_{\bl}$ are preserved. We postpone these calculations to the end of the proof below and assume for now that we do obtain an algebra homomorphism $\phi$.

From the choice of reduced expressions we made above, we have that $\phi$ maps $\mathcal{D}$ to the set 
$$\mathcal{B}^\prime=\{y_1^{m_1}\cdots y_n^{m_n}\psi_{\iota(w)}\idmep \mid w \in W(D_n), (m_1,\ldots,m_n) \in (\mathbb{Z}_{\geq 0})^n, \textbf{i}\in \beta \}\ .$$ 
This set $\mathcal{B}^\prime$ is a subset of the basis $\mathcal{B}$ of $V_{\bl}^{\beta}$ appearing in Theorem \ref{basis-theorem1}. Further it is immediate to see that, for $x\in\mathcal{B}$, we have $\rho(x)=x$ if $x\in\mathcal{B}^\prime$ while $\rho(x)=-x$ if $x\notin \mathcal{B}^\prime$. Since we are in characteristic different from 2 we conclude that $\mathcal{B}^\prime$ is a basis of $\bigl(V_{\bl}^{\beta}\bigr)^{\rho}$.

So we now conclude first that $\mathcal{D}$ is a linearly independent set since it is sent to a subset of a basis and is therefore a basis of $W^\beta_{\bl}$. Second, since $\phi$ maps a basis of $W^\beta_{\bl}$ bijectively to a basis of $\bigl(V_{\bl}^{\beta}\bigr)^{\rho}$, then it is bijective and the two algebras are isomorphic.

It remains to prove that we do obtain an algebra homomorphism $\phi$ by checking the defining relations (\ref{Rel:V1D})--(\ref{Rel:V11D}) of $W^\beta_{\bl}$. These are straightforward calculations using the defining relations (\ref{Rel:V1})--(\ref{Rel:V11}) of $V_{\bl}^{\beta}$. We give some of the details.

First, all relations not involving $\Psi_0$ are immediately seen to be preserved.

$\bullet$ \textbf{Relation (\ref{Rel:V3D}) with $a=0$.} We have
$$ \psi_0\psi_1\psi_0e(\ti)=e(r_0r_1r_0(\ti))\psi_0\psi_1\psi_0\,,$$
and so the relation is preserved since $s_0(\ti)=r_0r_1r_0(\ti)$.

$\bullet$ \textbf{Relation (\ref{Rel:V6D}).} If $j\notin\{1,2\}$ we have immediately $\psi_0\psi_1\psi_0y_j=y_j\psi_0\psi_1\psi_0$. Let $j=2$. Then we have
\[\psi_0\psi_1\psi_0y_2e(\ti)=\psi_0\psi_1y_2e(r_0(\ti))\psi_0e(\ti)=\psi_0\bigl(y_1\psi_1+\delta_{i_1^{-1},i_2}\bigr)\psi_0e(\ti)=(-y_1\psi_0\psi_1\psi_0+\delta_{i_1^{-1},i_2})e(\ti)\ .\]
The case $j=1$ is checked either by a similar calculation or by using the antiautomorphism of $V_{\bl}^{\beta}$ sending each generator to itself.

$\bullet$ \textbf{Relation (\ref{Rel:V7D}).} This relation is immediately verified if $b>2$, while for $b=1$ it follows from $\psi_0\psi_1\psi_0\psi_1=\psi_1\psi_0\psi_1\psi_0$.

$\bullet$ \textbf{Relation (\ref{Rel:V8D}).} We have, using $\psi_0^2=1$ and $\psi_0e(\ti)=e(r_0(\ti))\psi_0$:
\[(\psi_0\psi_1\psi_0)^2e(\ti)=\psi_0\psi_1^2\psi_0e(\ti)=\psi_0\cdot \psi_1^2e(r_0(\ti))\cdot\psi_0\ .\]
We note that $r_0(\ti)$ starts as $(i_1^{-1},i_2,\dots)$ and therefore the verification of (\ref{Rel:V8D}) is obtained directly by using (\ref{Rel:V6}) with $b=1$  followed by $\psi_0y_1=-y_1\psi_0$ and $\psi_0^2=1$.

$\bullet$ \textbf{Relation (\ref{Rel:V9D}).} We have, using $\psi_0\psi_2=\psi_2\psi_0$, $\psi_0^2=1$ and $\psi_0e(\ti)=e(r_0(\ti))\psi_0$:
\[(\psi_0\psi_1\psi_0\psi_2\psi_0\psi_1\psi_0-\psi_2\psi_0\psi_1\psi_0\psi_2)e(\ti)=\psi_0\cdot(\psi_1\psi_2\psi_1-\psi_2\psi_1\psi_2)e(r_0(\ti))\cdot\psi_0\ .\]
We note that $r_0(\ti)$ starts as $(i_1^{-1},i_2,i_3,\dots)$ and therefore (\ref{Rel:V9D}) is obtained directly by using (\ref{Rel:V7}) with $b=1$  followed by $\psi_0y_1=-y_1\psi_0$ and $\psi_0^2=1$.
\end{proof}

\begin{remark}
Theorem \ref{prop:fixed point VV is SVV} and its proof also shows that $W^\beta_{\bl}$ is isomorphic to the subalgebra of $V_{\bl}^{\beta}$ generated by $\psi_0\psi_1\psi_0$ together with all the generators of $V_{\bl}^{\beta}$ but $\psi_0$ (which coincides therefore with the subalgebra of fixed points $\bigl(V_{\bl}^{\beta}\bigr)^{\rho}$). An isomorphism is the natural map given in (\ref{iso-fixed-points}).\hfill$\triangle$
\end{remark}

\section{Cyclotomic quotients of $\hH(B_n)$ and $V_{\bl}^{\beta}$}\label{sec-cyc}

Let $n\geq 1$ and let $\bl$ and $S_{\bl}$ be as in Section \ref{sec-defV}. We fix a map (we call it a \emph{multiplicity map})
\begin{equation}\label{def-m}
m\ :\ S_{\bl}\to \Z_{\geq0}\ ,
\end{equation}
such that only a finite number of multiplicities $m(i)$ are different from $0$.
\begin{definition}
We define the cyclotomic quotient denoted $H(B_n)_{\bl,m}$ to be the quotient of the algebra $\hH(B_n)$ over the relation
\begin{equation}\label{cyc-rel}
\prod_{i\in S_{\bl}}(X_1-i)^{m(i)}=0\ .
\end{equation}
\end{definition}

\begin{definition}
Let $\beta$ be a finite union of $W(B_n)$-orbits in $(S_{\bl})^n$. We define the cyclotomic quotient denoted by $V_{\bl,m}^{\beta}$ to be the quotient of $V_{\bl}^{\beta}$ by the relation
\begin{equation}\label{cycV}
y_1^{m(i_1)}\idmep=0 \ \ \ \textrm{ for every } \textbf{i}=(i_1,\dots,i_n) \in \beta\ .
\end{equation}
Let $\mathcal{O}$ be the set of orbits in $(S_{\bl})^n$ under the action of the Weyl group $W(B_n)$. We set:
\begin{equation}\label{direct-sums}
V_{\bl}^{n}:=\bigoplus_{\beta\in\mathcal{O}}V_{\bl}^{\beta}\ \ \ \ \text{and}\ \ \ \ V_{\bl,m}^{n}:=\bigoplus_{\beta\in\mathcal{O}}V_{\bl,m}^{\beta}\ .
\end{equation}
\end{definition}

We record here a fundamental lemma (the proof in \cite[Lemma 2.1]{BK} can be repeated verbatim).
\begin{lemma}\label{lem-nil}
The elements $y_k\in V_{\bl,m}^n$, $k=1,\dots,n$, are nilpotent.
\end{lemma}

\begin{remark}\label{rem-mult}
Conjugating Relation (\ref{cycV}) by $\psi_0$, we obtain that $y_1^{m(i_1)}e(r_0(\ti))=0$ for all $\ti\in\beta$. This implies at once that the algebra $V_{\bl,m}^{\beta}$ is isomorphic to $V_{\bl,m'}^{\beta}$, where the new multiplicity map is given by $m'(i):=\text{min}\{m(i),m(i^{-1})\}$ for all $i\in S_{\bl}$. We have also $V_{\bl,m}^{n}$ isomorphic to $V_{\bl,m'}^{n}$.

Similarly, conjugating Relation (\ref{cyc-rel}) by $g_0$ and considering the minimal polynomial annihilating $X_1$ in $H(B_n)_{\bl,m}$, one obtains that $H(B_n)_{\bl,m}$ is in fact isomorphic to $H(B_n)_{\bl,m'}$ with $m'$ as before. Alternatively, this follows immediately from the preceding remark and Theorem \ref{theorem:type B iso when p^2=1}.

As a consequence of this remark, one could restrict the study of cyclotomic quotients to the situation where the multiplicity map $m$ satisfies $m(i^{-1})=m(i)$ for all $i\in S_{\bl}$. We will not do so as it is not really needed for our considerations (see Remark \ref{rem-mult2} though).

We also stress that this phenomenon is particular to the situation $p^2=1$. We refer to \cite[Proposition 3.2 and Remark 4.1]{PW} for similar considerations, with a different conclusion, in the situation $p^2\neq 1$.
\hfill$\triangle$
\end{remark}

\begin{remark}\label{rem-defcyc}
In the same spirit as in the discussion in Section \ref{sec-defV} leading to Formula (\ref{dec-V}), it is easy to see that the algebra $V_{\bl}^{n}$ can also be defined by generators and relations. The generators are as in Definition \ref{definition:vv algebra} with $\ti$ varying now over $(S_{\bl})^n$. The relations are as in Definition \ref{definition:vv algebra} with $\ti$ taken in $(S_{\bl})^n$ and with the first relation in (\ref{Rel:V1}) removed. The algebra $V_{\bl}^{n}$ is not unital if $S_{\bl}$ is not a finite set.

With this definition of $V_{\bl}^{n}$, the algebra $V_{\bl,m}^{n}$ can be seen as the quotient of $V_{\bl}^{n}$ over the relations (\ref{cycV}) with $\ti$ now varying over $(S_{\bl})^n$.\hfill$\triangle$
\end{remark}

\paragraph{Idempotents and blocks of $H(B_n)_{\bl, m}$.} In the same way as in \cite[Section 4.2]{PW}, we remark that $H(B_n)_{\bl, m}$ contains a finite set of non-trivial idempotents. Namely, we consider $H(B_n)_{\bl, m}$ as a finite-dimensional representation (by left multiplication) of the subalgebra of $H(B_n)_{\bl, m}$ generated by the commuting elements $X_1, \ldots, X_n$. From \cite[Proposition 3.1]{PW}, we have that all the eigenvalues of $X_1, \ldots, X_n$ belong to $(S_{\bl})^n$. We can decompose this representation into a direct sum of common generalised eigenspaces for the $X_i$:
\begin{equation*}
H(B_n)_{\bl, m}=\bigoplus_{\ti\in (S_{\bl})^n} M_{\ti}
\end{equation*}
where $M_{\ti}:=\{x\in H(B_n)_{\bl, m}\,\vert\ (X_k-i_k)^Nx=0\,,\ k=1,\dots,n\,,\ \text{for some $N>0$}\}$. We let 
\begin{equation*}
\{ e_{\ti}^H \}_{\ti \in (S_{\bl})^n}
\end{equation*}
denote the associated set of mutually orthogonal idempotents which, by definition, belong to the commutative subalgebra generated by $X_1, \ldots, X_n$.

Fora $W(B_n)$-orbit $\beta$ in $(S_{\bl})^n$, we set 
\begin{equation*}
e_\beta := \sum_{\ti \in \beta} e^H(\ti).
\end{equation*}
The element $e_\beta$ is therefore a central idempotent in $H(B_n)_{\bl, m}$. Therefore the set $e_\beta H(B_n)_{\bl, m}$ is either $\{0\}$ or an idempotent subalgebra (with unit $e_\beta$) which is a union of blocks of $H(B_n)_{\bl, m}$ and we have, similarly to (\ref{direct-sums});
\[H(B_n)_{\bl, m}=\bigoplus_{\beta\in\mathcal{O}}e_\beta H(B_n)_{\bl, m}\ .\]

\begin{theorem}\label{theorem:type B iso when p^2=1}
The algebras $H(B_n)_{\bl,m}$ and $V_{\bl,m}^{n}$ are isomorphic. More precisely, the algebras $e_\beta H(B_n)_{\bl, m}$ and $V^\beta_{\bl, m}$ are isomorphic, for any orbit $\beta\in\mathcal{O}$.
\end{theorem}
\begin{proof}
We recall that we have $p^2=1$, and thus in particular we have the relations $g_0^2=1$, $g_0X_1=X_1^{-1}g_0$ in $\hH(B_n)$.
\\
To prove the theorem we prove here that $e_\beta H(B_n)_{\bl, m}$ and $V^\beta_{\bl, m}$ are isomorphic, for any orbit $\beta\in\mathcal{O}$. The theorem follows from this by taking direct sums over all orbits $\beta\in\mathcal{O}$. We will use similar constructions and arguments as in \cite{PW} which were inspired from \cite{BK}. Fortunately, many calculations have been done in \cite{PW} (which treated the situation $p^2\neq 1$) and can be used here as well. However, necessary modifications are needed when $p^2=1$. The resulting algebra $V_{\bl,m}^\beta$ is different and one cannot simply take $p^2=1$ in \cite{PW}.

\vskip .2cm
First, we lay out some notation. We let $f \in K[[z]]$ be the formal power series
\begin{equation*}
f(z):=z+\frac{z}{1-z}.
\end{equation*}
Then, since the degree 0 coefficient is zero and the degree 1 coefficient is non-zero, $f$ has a composition inverse $g=\sum_i b_iz^i\in K[[z]]$ with $b_0=0$ and $b_1\neq 0$. That is, $$f\bigl(g(z)\bigr)=g\bigl(f(z)\bigr)=z\ .$$

Next, we label elements of both $V^\beta_{\bl,m}$ and $H(B_n)_{\bl,m}$ in a way in which, at first, will appear redundant but which will be needed later. We set
\begin{equation*}
y_k^V:=y_k\ \ \text{and}\ \ e^{V}(\ti):=e(\ti)\ \ \ \text{in $V^\beta_{\bl,m}$,}\ \ \ \qquad\text{and}\ \ \ \qquad X_k^H:=X_k\ \ \ \text{in $H(B_n)_{\bl,m}$.}
\end{equation*}
for all $k=1,\ldots, n$, all $\ti\in\beta$ and all $s=0,1,\ldots,n-1$.

\paragraph{\textbf{Commuting elements.}} Let $k\in\{1,\dots,n\}$. We define elements of $H(B_n)_{\bl, m}$:
\begin{equation}\label{def-yH}
y^H_i:=\displaystyle\sum_{\ti \in \beta}f(1-i_k^{-1}X_k) e^H(\ti)=\displaystyle\sum_{\ti \in \beta}(i_kX_k^{-1}-i_k^{-1}X_k) e^H(\ti)\ .
\end{equation}
These elements are well-defined since $(1-i_k^{-1}X_k)e^H(\ti)$ is nilpotent for any $\ti\in\beta$, by definition of the idempotents $e^H(\ti)$. Moreover, the elements $y^H_k$ commute and they are nilpotent since the formal power series $f$ has no constant term.

Note that by definition of $f$ and $g$ we have then that $X_k=\sum_{\textbf{i} \in \beta}i_k\bigl(1-g(y^H_k)\bigr)e(\ti)$ in $H(B_n)_{\bl, m}$. Therefore, conversely, we define elements of $V^\beta_{\bl,m}$;
\begin{equation}\label{def-XV}
X_k^V:=\sum_{\textbf{i} \in \beta}i_k\bigl(1-g(y_k)\bigr)\idmep\in V^\beta_{\bl,m}\ .
\end{equation}
These elements are well-defined since the elements $y_k$ are nilpotent. Moreover, the elements $X_k^V$ commute and they are invertible since $g$ has no constant term.

\paragraph{\textbf{Notations.}} 
%
%
$\bullet$ We introduce a symbol $\some$ which we will use as an upper index to denote both $V$ and $H$. Let $k\in\{1,\dots,n\}$ and $i\in S_{\bl}$. We set
\begin{equation}\label{nota1}
X^{\some}_k(i):=i\bigl(1-g(y^{\some}_k)\bigr)\ ,
\end{equation}
so that we have $X^{\some}_ke^{\some}(\ti)=X^{\some}_k(i_k)e^{\some}(\ti)$ for any $\ti\in \beta$.

\vskip .1cm $\bullet$ Let $Y^{\some}\in K[[y^{\some}_1,\dots,y^{\some}_n]$. The element $Y^V$ is a well-defined element of $V^\beta_{\bl,m}$ since $y^{V}_1,\dots,y^{V}_n$ commute and are nilpotent by Lemma \ref{lem-nil}, and the element $Y^H$ is a well-defined element of $H(B_n)_{\bl,m}$ since $y^{H}_1,\dots,y^{H}_n$ commute and are nilpotent by construction. Such an element $Y^{\some}$ is invertible if and only if its 0 degree coefficient is non-zero. In the following, when we write a rational fraction in $y^{\some}_1,\dots,y^{\some}_n$, we mean (often implicitly) that we have checked that it is well-defined.

For example, the element $X^{\some}_k(i)$ of the preceding item is invertible since $g$ has no degree 0 term.

Another example used explicitly later is the following. Let $k\in\{1,\dots,n-1\}$ and $i,j\in S_{\bl}$. Then the element $(1-X_k^\some(i)X_{k+1}^\some(j)^{-1})$ is invertible whenever $i\neq j$.

\paragraph{\textbf{Renormalisation elements.}} We fix a family of invertible power series $Q_k(i,j) \in K[[z, z']]$, where $k=1,\dots,n-1$ and $i,j\in S_{\bl}$, and we set for $\some \in \{ V, H \}$:
$$Q^{\some}_k(i,j):=Q_k(i,j)\bigl(y^{\some}_k,y^{\some}_{k+1}\bigr) \in K[[y^{\some}_k,y^{\some}_{k+1}]]\ .$$

Recall that the Weyl group $W(B_n)$ acts on a formal power series $f$ in $y^{\some}_1,\dots,y^{\some}_n$ (for $k=1,\dots,n-1$, $r_k$ acts by exchanging $y^{\some}_k$ and $y^{\some}_{k+1}$ and $r_0$ acts by multiplying $y_1^\some$ by $-1$),  and we denote ${}^wf$ the action of $w\in W(B_n)$ on $f$. 

We assume that the following properties are satisfied for any $i,j\in S_{\bl}$ and $k=1,\dots,n-1$,
\begin{equation}\label{eq-Q1}
Q_k^{\some}(i,j)=i^{-1}\big( q^{-1}X^{\some}_k(i)-qX^{\some}_{k+1}(j) \big)\displaystyle\frac{y^{\some}_{k+1}-y^{\some}_k}{g(y^{\some}_{k+1})-g(y^{\some}_k)}\ \ \ \ \ \textrm{ if } i=j\,,
\end{equation}
\begin{equation}\label{eq-Q2}
{^{r_k}}Q^{\some}_k(j,i)Q^{\some}_k(i,j)=\left\lbrace
\begin{array}{ll}
\Gamma_k^{\some}(i,j) & \textrm{ if }i \nleftrightarrow j \vspace*{0.6em}\\
\Gamma_k^{\some}(i,j)\big( y_{k+1}^{\some}-y_k^{\some} \big)^{-1} & \textrm{ if }i \leftarrow j \vspace*{0.6em}\\
\Gamma_k^{\some}(i,j)\big( y_{k}^{\some}-y_{k+1}^{\some} \big)^{-1} & \textrm{ if }i \rightarrow j \vspace*{0.6em}\\
\Gamma_k^{\some}(i,j)\big( y_{k+1}^{\some}-y_k^{\some} \big)^{-1}\big( y_k^{\some}-y_{k+1}^{\some} \big)^{-1} & \textrm{ if }i \leftrightarrow j 
\end{array}
\right.
\end{equation}
and, for $k=1,\dots,n-2$,
\begin{equation}\label{eq-Q3}
{^{r_k}}Q^{\some}_{k+1}(i,j)={^{r_{k+1}}}Q^{\some}_k(i,j)\ \ \ \ \text{and}\ \ \ \ Q_1(i,j)={}^{r_0r_1r_0}Q_1(j^{-1},i^{-1})\ ,
\end{equation}
where 
\begin{equation*}
\Gamma_k(i,j) = \frac{(qX^\some_k(i)-q^{-1}X^\some_{k+1}(j))(qX^\some_{k+1}(j)-q^{-1}X^\some_k(i))}{(X^\some_k(i)-X^\some_{k+1}(j))(X^\some_{k+1}(j)-X^\some_k(i))}.
\end{equation*}

By \cite[Lemma 7.3]{PW}, such a family of invertible power series satisfying Conditions (\ref{eq-Q1})--(\ref{eq-Q3}) exists, and can be chosen explicitly.

\paragraph{\textbf{Algebra homomorphisms.}} We also set, for $k=1,\dots,n-1$ and $i,j\in S_{\bl}$,
\begin{equation}\label{def-Pk}
P_k^\some(i,j):= \left\lbrace
\begin{array}{ll}
q^{-1} & \textrm{ if }i=j\,, \\
-(q-q^{-1})(1-X_k^\some(i)X_{k+1}^\some(j)^{-1})^{-1} & \textrm{ if }i\neq j\ .
\end{array}
\right.
\end{equation}

Now in $e_{\beta}H(B_n)_{\blambda,m}$ we set, where $k=1,\dots,n-1$,
\begin{equation}\label{def-psiH}
\psi_0^H:=g_0e_{\beta}\,,\ \ \ \ \ \psi_k^H := \displaystyle\sum_{\ti \in \beta} \Bigl(g_k+P_k^H(i_k,i_{k+1})\Bigr) \Bigl(Q^H_k(i_k,i_{k+1})\Bigr)^{-1} e^H(\ti)\ ,
\end{equation}
and conversely in $V^\beta_{\bl, m}$, 
\begin{equation}\label{def-gV}
g_0^V:=\psi_0\,,\ \ \ \ \ g_k^{V}:=\displaystyle\sum_{\ti \in \beta}\Bigl(\psi_kQ_k^V(i_k,i_{k+1}) - P_k^V(i_k,i_{k+1})\Bigr)\idmep\ .
\end{equation}
Finally, using the definitions we made in (\ref{def-yH}), (\ref{def-XV}) and (\ref{def-psiH})--(\ref{def-gV}), we can introduce the following maps, where $j\in\{1,\dots,n\}$, $k\in\{0,\dots,n-1\}$ and $\ti\in\beta$:
\begin{equation}\label{iso-maps}
\rho\ :\ \ \begin{array}{rcl}
V^\beta_{\bl, m} & \to & e_{\beta}H(B_n)_{\blambda,m}\\[0.5em]
e(\ti) & \mapsto & e^H(\ti)\\[0.2em]
y_j & \mapsto & y^H_j\\[0.2em]
\psi_k & \mapsto & \psi_k^H
\end{array}\ 
\ \ \quad \text{and}\quad \ \ \ 
\sigma\ :\ \ \begin{array}{rcl}
\hH(B_n) & \to & V^\beta_{\bl, m}\\[0.5em]
X_i & \mapsto & X_i^V\\[0.2em]
g_k & \mapsto & g_k^{V}
\end{array}\ .
\end{equation}
Assuming for a moment that these two maps extend to algebra homomorphisms, still denoted by $\rho$ and $\sigma$, we will show that $\rho$ has a 2-sided inverse. This inverse is obtained from the fact that $\sigma$ factors through the natural map $\hH(B_n) \twoheadrightarrow e_\beta H(B_n)_{\bl, m}$ which is obtained by composing the natural surjection to the quotient with multiplication by $e_\beta$. First, we must verify that the cyclotomic relation (\ref{cyc-rel}) holds in $V^\beta_{\bl, m}$. We have, using that $\sigma$ is an algebra morphism and the orthogonality of the idempotents in $V^\beta_{\bl,m}$,
\begin{equation*}
\sigma \left(\prod_{j \in S_{\bl}}(X_1-j)^{m(j)}\right)=\sum_{\ti \in \beta}\left(\prod_{j \in S_{\bl}}\big(i_1(1-g(y_1))-j\big)^{m(j)}\idmep\right).
\end{equation*}
For each $\ti\in\beta$ take $j=i_1$ in the product above and note that $g(y_1)^{m(i_1)}e(\ti)=0$ since $g$ has no constant term and $(y_1)^{m(i_1)}e(\ti)=0$ by definition in (\ref{cycV}). So every summand in the above expression has a factor which vanishes. Therefore $\sigma$ factors through $\hH(B_n) \twoheadrightarrow H(B_n)_{\bl, m}$. To show that $\sigma$ can be further factored through $H(B_n)_{\bl,m} \twoheadrightarrow e_\beta H(B_n)_{\bl, m}$ we must show that $\sigma(e^H(\ti))=0$ whenever $\ti \not\in \beta$. This follows the same reasoning as in \cite[Lemma 3.4]{BK} or \cite[Lemma 6.9]{PW} and we do not repeat the details.

Finally, abusing notation, let us still denote by $\sigma$ the morphism $e_\beta H(B_n)_{\bl, m} \longrightarrow V^\beta_{\bl,m}$. We check that $\sigma \circ \rho = \textrm{Id}_{e_\beta H(B_n)_{\bl, m}}$ and $\rho\circ\sigma = \textrm{Id}_{V^\beta_{\bl,m}}$ by easily verifying these equalities on the generators.

\vskip .2cm
To conclude the proof of the theorem, it remains to show that the maps $\rho$ and $\sigma$ do extend to algebra homomorphisms. Note that the relations of $V^\beta_{\bl,m}$ not involving the generator $\psi_0$ are the same as the ones appearing in \cite{PW}, and the maps $\rho$ and $\sigma$ restricted to all the generators but $\psi_0$ and $g_0$ coincide with the corresponding ones in \cite{PW}. So in fact many relations have been shown to be preserved by $\rho$ and $\sigma$ in \cite{PW}. It remains only to consider the relations involving $g_0$ for $\sigma$, and the relations involving $\psi_0$ for $\rho$.

\paragraph{\textbf{$\sigma$ is an algebra homomorphism.}} It suffices to verify that the relations (\ref{rel-H1}), (\ref{rel-H2}), (\ref{rel-H4}), (\ref{rel-H6}), (\ref{rel-H8}) are preserved, where $i=0$ in (\ref{rel-H1}), (\ref{rel-H4}), (\ref{rel-H8}).
\\[0.2cm]
\textbf{Relation (\ref{rel-H1}) with $i=0$}. Recall that $q_0=p$ and that $p^2=1$. We have directly $(g_0^V)^2 = \psi_0^2 = 1$.
\\[0.2cm]
\textbf{Relations (\ref{rel-H4}) and (\ref{rel-H8}) with $i=0$}. We have in $V^\beta_{\bl,m}$ that $\psi_0$ commutes with $\psi_j$ and $y_j$ if $j>1$. Moreover, in Formulas (\ref{def-XV}) and (\ref{def-gV}), the coefficients in front of $e(\ti)$ and $e(r_0(\ti))$ are equal. It follows that $g_0^V=\psi_0$ commutes with $X_j^V$ and $g_j^V$ for $j>1$.
\\[0.2cm]
\textbf{Relation (\ref{rel-H6})}. To check that $\sigma$ preserves this relation we require Lemma 7.2 from \cite{PW}. This states that the formal power series $g$ satisfies
\begin{equation*}
\frac{1}{(1-g(z))}=1-g(-z).
\end{equation*}
We have:
\begin{equation*}
\begin{split}
g_0^V(X_1^V)^{-1}g_0^V\idmep &= \psi_0i_1(1-g(y_1))^{-1}\psi_0\idmep \\
&= \psi_0i_1(1-g(-y_1))\psi_0\idmep \\
&=i_1(1-g(y_1))\idmep \\
&=X^V_1\idmep.
\end{split}
\end{equation*}
Here we have used the defining relations $\psi_0^2=1$ and $\psi_0y_1=-y_1\psi_0$.
\\[0.2cm]
\textbf{Relation (\ref{rel-H2})}. We will check this relation in two steps. First note that, from (\ref{Rel:V4}), we have in $V^\beta_{\bl,m}$
\begin{equation}\label{commut1}
(\psi_1F - {^{r_1}}F\psi_1)\idmep=\delta_{i_1,i_{2}}\frac{F-{^{r_1}}F}{y_{2}-y_1}\idmep
\end{equation}
for any $F\in K[[y_1,\dots,y_n]]$ and $\ti\in\beta$. We also have $\psi_0F={^{r_0}}F\psi_0$ from (\ref{Rel:V8}).

In $V^\beta_{\bl,m}$, we set 
$$\Phi^V_1:=\displaystyle\sum_{\ti\in\beta} \psi_1Q^V_1(i_1,i_2)e(\ti)=\displaystyle\sum_{\ti\in\beta}\Bigl(g^V_1+P_1^V(i_1,i_2)\Bigr)e(\ti)\ .$$
\\[0.1cm]
\textbf{Step 1.} The first step is to show the relation $\psi_0\Phi^V_1\psi_0\Phi^V_1=\Phi^V_1\psi_0\Phi^V_1\psi_0$. We remove the superscript $V$ for brevity.

We take $\ti\in\beta$ and set $i:=i_1$ and $j:=i_2$. We have:
\begin{multline*}
\bigl(\psi_0\Phi_1\psi_0\Phi_1-\Phi_1\psi_0\Phi_1\psi_0\bigr)e(\ti)=\Bigl(\psi_0\psi_1Q_1(j^{-1},i)\psi_0\psi_1Q_1(i,j)-\psi_1Q_1(j^{-1},i^{-1})\psi_0\psi_1Q_1(i^{-1},j)\psi_0\Bigr)e(\ti)\\
=\psi_0\psi_1\psi_0\cdot {}^{r_0}Q_1(j^{-1},i)\psi_1e(\ti)\cdot Q_1(i,j) - \psi_1\psi_0\cdot {}^{r_0}Q_1(j^{-1},i^{-1})\psi_1e(r_0(\ti))\cdot \psi_0{}^{r_0}Q_1(i^{-1},j)\ .
\end{multline*}
Now, if $i=j$, we find using Condition (\ref{eq-Q3}) that ${}^{r_1r_0}Q_1(j^{-1},i)={}^{r_0}Q_1(j^{-1},i)$. Therefore, using (\ref{commut1}), we have ${}^{r_0}Q_1(j^{-1},i)\psi_1e(\ti)=\psi_1{}^{r_1r_0}Q_1(j^{-1},i)e(\ti)$ for every $\ti\in\beta$. 

A similar reasoning shows that we also have ${}^{r_0}Q_1(j^{-1},i^{-1})\psi_1e(r_0(\ti))=\psi_1{}^{r_1r_0}Q_1(j^{-1},i^{-1})e(r_0(\ti))$ for every $\ti\in\beta$.
We conclude that:
\begin{multline*}
\bigl(\psi_0\Phi_1\psi_0\Phi_1-\Phi_1\psi_0\Phi_1\psi_0\bigr)e(\ti)\\
=\Bigl(\psi_0\psi_1\psi_0\psi_1{}^{r_1r_0}Q_1(j^{-1},i)Q_1(i,j)-\psi_1\psi_0\psi_1\psi_0{}^{r_0r_1r_0}Q_1(j^{-1},i^{-1}){}^{r_0}Q_1(i^{-1},j)\Bigr)e(\ti)\ .
\end{multline*}
We again use Condition (\ref{eq-Q3}) to see that the coefficients ${}^{r_1r_0}Q_1(j^{-1},i)Q_1(i,j)$ and ${}^{r_0r_1r_0}Q_1(j^{-1},i^{-1}){}^{r_0}Q_1(i^{-1},j)$ are in fact equal, and this concludes step 1.
\\[0.1cm]
\textbf{Step 2.} To conclude, we must show that $\psi_0g^V_1\psi_0g^V_1=g^V_1\psi_0g^V_1\psi_0$, which is equivalent to showing
\begin{equation}\label{psiPhi}
\psi_0\bigl(\Phi_1-P_1(j^{-1},i)\bigr)\psi_0\bigl(\Phi_1-P_1(i,j)\bigr)e(\ti)=\bigl(\Phi_1-P_1(j^{-1},i^{-1})\bigr)\psi_0\bigl(\Phi_1-P_1(i^{-1},j)\bigr)\psi_0e(\ti)\ .
\end{equation}
From the definition of $P_1$ in (\ref{def-Pk}), a short straightforward verification gives that, for every $i',j'\in S_{\bl}$,
\begin{equation}\label{eq:P-property-1}
P_1(i',j')={^{r_0r_1r_0}}P_1(i',j')=P_1(j'^{-1},i'^{-1})={^{r_0r_1r_0}}P_1(j'^{-1},i'^{-1})\ .
\end{equation} 
This shows in particular that ${}^{r_0}P_1(i',j')$ is invariant under the action of $r_1$ and thus commutes with $\psi_1$ for any $i',j'\in S_{\bl}$. So expanding the left hand side of (\ref{psiPhi}), we find (we omit the idempotent $e(\ti)$ here):
\[\psi_0\Phi_1\psi_0\Phi_1-\Phi_1{}^{r_0}P_1(j^{-1},i)-\psi_0\Phi_1\psi_0P_1(i,j)+{}^{r_0}P_1(j^{-1},i)P_1(i,j)\ ,\]
while the right hand side gives:
\[\Phi_1\psi_0\Phi_1\psi_0-\Phi_1{}^{r_0}P_1(i^{-1},j)-\psi_0\Phi_1\psi_0{}^{r_0r_1r_0}P_1(i,j)+{}^{r_0r_1r_0}P_1(i,j){}^{r_0}P_1(i^{-1},j)\ .\]
We conclude, using step 1 and the equalities (\ref{eq:P-property-1}), that these two expressions are equal. 

This concludes the verification of Relation (\ref{rel-H2}) and in turn the verification that $\sigma$ extends to an algebra homomorphism.

\paragraph{\textbf{$\rho$ is an algebra morphism.}} It suffices to prove that only the relations involving $\psi_0$ are preserved so we will now verify that relations (\ref{Rel:V3}) and (\ref{Rel:V5}) with $a=0$, and (\ref{Rel:V8})--(\ref{Rel:V11}) are preserved by $\rho$.
\\[0.2cm]
\textbf{Relation (\ref{Rel:V3}) for $a=0$}. Recall that in $H(B_n)_{\bl, m}$, since $p^2=1$, we have $g_0P={}^{r_0}Pg_0$ for any $P\in K[X_1^{\pm1},\dots,X_n]$, where $r_0$ acts on $P$ by inverting $X_1$. This means that if $x\in H(B_n)_{\bl, m}$ is such that $Px=0$ then $g_0x$ is annihilated by ${}^{r_0}P$. Then from the definition of the idempotents $e^H(\ti)$ of $H(B_n)_{\bl, m}$, this means that we have $g_0e^H(\ti)=e^H(r_0(\ti))g_0$, which is the relation required for checking (\ref{Rel:V3}) .
\\[0.2cm]
\textbf{Relations (\ref{Rel:V5}) for $a=0$ and (\ref{Rel:V9})}. In $H(B_n)_{\bl,m}$, we have that $g_0$ commutes with $g_b$ and $X_b$ if $b>1$. Moreover, in (\ref{def-yH}) and (\ref{def-psiH}), the coefficients in front of $e(\ti)$ and $e(r_0(\ti))$ are the same. It follows that $\psi^H_0=g_0$ commutes with $y^H_b$ and with $\psi^H_b$ if $b>1$.
\\[0.2cm]
\textbf{Relations (\ref{Rel:V8})}. For relation (\ref{Rel:V8}) we have, using the explicit form of $f$, 
\begin{equation*}
(\psi_0^Hy_1^H+y_1^H\psi_0^H)e^H(\ti) = (g_0(i_1X_1^{-1}-i_1^{-1}X_1) + (i_1^{-1}X_1^{-1}-i_1X_1)g_0)e^H(\ti) = 0.
\end{equation*}
We have used that $g_0X_1^{\pm1}=X_1^{\mp1}g_0$ in $H(B_n)_{\bl, m}$.
\\[0.2cm]
\textbf{Relation (\ref{Rel:V10})} is directly seen to be preserved since $(\psi^H_0)^2=g_0^2=1$.
\\[0.2cm]
\textbf{Relation (\ref{Rel:V11})}. In $e_{\beta}H(B_n)_{\bl, m}$, we set 
$$\Phi^H_1:=\displaystyle\sum_{\ti\in\beta}(g_1+P^H_1(i_1,i_2))e^H(\ti)=\displaystyle\sum_{\ti\in\beta}\psi^H_1Q_1^H(i_1,i_2)e^H(\ti)\ .$$
\\[0.1cm]
\textbf{Step 1.} The first step is to show the relation $g_0\Phi^H_1g_0\Phi^H_1=\Phi^H_1g_0\Phi^H_1g_0$. For this step, we replace $\Phi_1^H$ by its expression in terms of $g_1$ and $P_1$ and we use the exact same reasoning as in step 2 of the verification of Relation (\ref{rel-H2}), with $g_0,g_1$ replacing $\psi_0, \Phi_1$.
\\[0.1cm]
\textbf{Step 2.} Then to conclude, we must show that $g_0\psi^H_1g_0\psi^H_1=\psi^H_1g_0\psi^H_1g_0$ which is equivalent to showing
\[\Bigl(g_0\Phi^H_1Q^H_1(j^{-1},i)^{-1}g_0\Phi^H_1Q_1(i,j)^{-1}-\Phi^H_1Q^H_1(j^{-1},i^{-1})^{-1}g_0\Phi^H_1Q^H_1(i^{-1},j)^{-1}g_0\Bigr)e^H(\ti)\,,\]
for any $\ti\in\beta$, where we set $(i,j):=(i_1,i_2)$. To do this, we use the exact same reasoning as in step 1 of the verification of Relation (\ref{rel-H2}), with $g_0, \Phi^H_1$ replacing $\psi_0, \psi_1$. We do not repeat it, we only point out that at this point we know that we have in $H(B_n)_{\bl, m}$
\[
(\psi^H_1F - {^{r_1}}F\psi^H_1)e^H(\ti)=\delta_{i_1,i_{2}}\frac{F-{^{r_1}}F}{y^H_{2}-y^H_1}e^H(\ti)
\]
for any $F\in K[[y^H_1,\dots,y^H_n]]$ and $\ti\in\beta$, which is equivalent to
\[
(\Phi^H_1F - {^{r_1}}F\Phi^H_1)e^H(\ti)=\delta_{i_1,i_{2}}Q_1^H(i_1,i_2)\frac{F-{^{r_1}}F}{y^H_{2}-y^H_1}e^H(\ti)\ .
\]
This formula replaces Formula (\ref{commut1}) which was used in the reasoning in step 1 of the verification of Relation (\ref{rel-H2}).
\\[0.4cm]
This concludes the verification that $\rho$ extends to an algebra homomorphism and in turn the proof of the theorem.
\end{proof}

\section{Cyclotomic quotients of $\hH(D_n)$ and $W_{\bl}^{\delta}$}\label{sec-cycD}

Let $n\geq 2$ and let $\bl$ and $S_{\bl}$ be as in Sections \ref{sec-defV} and \ref{sec-cyc}. Take a multiplicity map $m$ as in (\ref{def-m}). We give definitions for type $D$ which are analogous to the ones given at the beginning of Section \ref{sec-cyc}.
\begin{definition}\label{def-cycD}
We define the cyclotomic quotient, denoted $H(D_n)_{\bl,m}$, to be the quotient of the algebra $\hH(D_n)$ over the relation
\begin{equation}\label{cyc-relD}
\prod_{i\in S_{\bl}}(X_1-i)^{m(i)}=0\ .
\end{equation}
\end{definition}

\begin{definition}\label{def-cyc-quivD}
Let $\delta$ be a finite union of $W(D_n)$-orbits in $(S_{\bl})^n$. We define the cyclotomic quotient, denoted by $W_{\bl,m}^{\delta}$, to be the quotient of $W_{\bl}^{\delta}$ by the relation
\begin{equation}\label{cycW}
y_1^{m(i_1)}\idmep=0 \ \ \ \textrm{ for every } \textbf{i}=(i_1,\dots,i_n) \in \delta\ .
\end{equation}
Let $\mathcal{O}'$ be the set of orbits in $(S_{\bl})^n$ under the action of the Weyl group $W(D_n)$. We set:
\begin{equation}\label{direct-sums2}
W_{\bl}^{n}:=\bigoplus_{\delta\in\mathcal{O}'}W_{\bl}^{\delta}\ \ \ \ \text{and}\ \ \ \ W_{\bl,m}^{n}:=\bigoplus_{\delta\in\mathcal{O}'}W_{\bl,m}^{\delta}\ .
\end{equation}
\end{definition}
Note that, by grouping some summands in the direct sums above (as explained in Remark \ref{rem-orb}), we also have
\[
W_{\bl}^{n}:=\bigoplus_{\beta\in\mathcal{O}}W_{\bl}^{\beta}\ \ \ \ \text{and}\ \ \ \ W_{\bl,m}^{n}:=\bigoplus_{\beta\in\mathcal{O}}W_{\bl,m}^{\beta}\ ,
\]
where $\mathcal{O}$ is the set of $W(B_n)$-orbits in $(S_{\bl})^n$.

\begin{remark}
A remark similar to Remark \ref{rem-defcyc} applies here as well about presenting $W_{\bl}^{n}$ directly by generators and relations and seeing $W_{\bl,m}^{n}$ directly as a quotient of $W_{\bl}^{n}$.\hfill$\triangle$
\end{remark}

\subsection{Cyclotomic quotients $H(D_n)_{\bl,m}$ and $W_{\bl,m}^{\delta}$ as subalgebras of fixed points.}

In this subsection, under some condition on the multiplicity map $m$, we relate the cyclotomic quotients $H(D_n)_{\bl,m}$ and $W_{\bl,m}^{n}$ to the corresponding cyclotomic quotients $H(B_n)_{\bl,m}$ and $V_{\bl,m}^{n}$. 

We will use the following general statement.
\begin{lemma}\label{general-lemma}
Let $A$ be a $K$-algebra equipped with an involutive automorphism $\theta$, and denote by $A^{\theta}$ the set of fixed points. Let $J$ be a two-sided ideal of $A$ such that $\theta(J)\subset J$, and denote by $J^{\theta}$ the set of fixed points of $J$. Then we have
\[A^{\theta}/J^{\theta}\cong (A/J)^{\tilde{\theta}}\ ,\]
where $\tilde{\theta}$ is the involutive automorphism of $A/J$ given by $\tilde{\theta}(x+J)=\theta(x)+J$.
\end{lemma}
\begin{proof}
First, we note that the map $\tilde{\theta}$ is well-defined because if $x-y\in J$ then $\theta(x)-\theta(y)=\theta(x-y)$ also belongs to $J$ since $J$ is stable under $\theta$. Then $\tilde{\theta}$ is clearly an involutive automorphism of $A/J$.

Now consider the map
\begin{equation*}
\begin{split}
\pi\ :\ A^\theta &\longrightarrow A/J \\
x &\mapsto x + J.
\end{split}
\end{equation*}
The map $\pi$ is the restriction to $A^\theta$ of the canonical surjective morphism from $A$ to its quotient $A/J$.

First we claim that the image of $\pi$ is the subalgebra of fixed points $(A/J)^{\tilde{\theta}}$. Indeed for one inclusion, if $x\in A^\theta$ then $\tilde{\theta}(x+J)=\theta(x)+J=x+J$.

For the reverse inclusion, take any $h+J \in (A/J)^{\tilde{\theta}}$. Then we have $\theta(h)+J=h+J$ which implies $\frac{1}{2}(\theta(h)+h)+J=h+J$ (the characteristic of $K$ is different from 2). So we have that $\frac{1}{2}(h+\theta(h)) \in A^\theta$ and its image under $\pi$ is $h+J$.

Finally, the kernel of $\pi$ is the intersection of $J$ with $A^\theta$. That is, we have $\textrm{ker}(\pi)=J^\theta$.
\end{proof}

\paragraph{\textbf{Cyclotomic quotients $H(D_n)_{\bl,m}$ as subalgebras of fixed points of $H(B_n)_{\bl,m}$.}} Recall the involutive automorphism $\eta$ of $\hH(B_n)$ defined by:
\begin{equation}\label{eta2}
\eta\ :\ \ \ \ g_0\mapsto -g_0\,,\ \ \ \ g_i\mapsto g_i\,,\ i=1,\dots,n-1,\ \ \ \ \ \ X^x\mapsto X^x\ .
\end{equation}
Let $J \subset \hH(B_n)$ be the two-sided ideal generated by the element
\begin{equation}\label{ele-quot}
P(X_1)=\prod_{i\in S_{\bl}} (X_1-i)^{m(i)}
\end{equation}
so that $H(B_n)_{\bl, m}= \hH(B_n)/J$ by definition. Clearly we have that $\eta(J)\subset J$ since the element (\ref{ele-quot}) is fixed by $\eta$.

Note that, slightly abusing notation, we use the same names for elements of the quotient $H(B_n)_{\bl,m}$ as for elements of $\hH(B_n)$. With this notation, the involutive automorphism $\tilde{\eta}$ of $H(B_n)_{\bl,m}$ appearing in the conclusion of Lemma \ref{general-lemma} is simply given by the same formulas:
\begin{equation}\label{tilde-eta}
\tilde{\eta}\ :\ \ \ \ g_0\mapsto -g_0\,,\ \ \ \ g_i\mapsto g_i\,,\ i=1,\dots,n-1,\ \ \ \ \ \ X^x\mapsto X^x\ .
\end{equation}

\begin{proposition}\label{prop-fixed-points1}
Assume that $m$ satisfies $m(i^{-1})=m(i)$ for any $i\in S_{\bl}$. Then the algebra $H(D_n)_{\bl, m}$ is isomorphic to $(H(B_n)_{\bl,m})^{\tilde{\eta}}$, with $\tilde{\eta}$ given in (\ref{tilde-eta}).
\end{proposition}
\begin{proof} First, we have by Lemma \ref{general-lemma} that $(H(B_n)_{\bl,m})^{\tilde{\eta}}$ is isomorphic to $\hH(B_n)^{\eta}/J^{\eta}$.

Besides, we can identify $\hH(D_n)$ with the subalgebra of fixed points $\hH(B_n)^{\eta}$ (see the end of Section \ref{sec-def}). Under this identification, the elements $X^x$ in $\hH(D_n)$ are sent to elements $X^x$ in $\hH(B_n)$. Thus from Definition \ref{def-cycD}, we have that $H(D_n)_{\bl, m}$ is the quotient of $\hH(B_n)^{\eta}$ by the ideal generated by the element $P(X_1)$ in (\ref{ele-quot}). Denote by $J'$ this ideal of $\hH(B_n)^{\eta}$. We must then prove that $J'=J^{\eta}$.

For the inclusion $J'\subset J^{\eta}$, an element of $J'$ is a linear combination of elements of the form $j=xP(X_1)y$ with $x,y\in \hH(B_n)^{\eta}$. Such an element $j$ lies obviously in $J$ and moreover $\eta(j)=j$ since $\eta\bigl(P(X_1)\bigr)=P(X_1)$ and $x,y\in \hH(B_n)^{\eta}$.

For the reverse inclusion, let $\mathcal{B}$ be the basis in (\ref{base-aff}) of $\hH(B_n)$ consisting of elements $X^xg_w$, with $x\in L$ and $w\in W(B_n)$. The set $\mathcal{B}$ splits into two disjoint subsets 
$$\mathcal{B}=\mathcal{B}^+\cup\mathcal{B}^-\,,\ \ \ \ \ \text{where $\mathcal{B}^{\pm}=\{b\in\mathcal{B}\ |\ \eta(b)=\pm b\}$.}$$
The subset $\mathcal{B}^+$ is formed by the elements $X^xg_w$ with $w$ in the subgroup $W(D_n)$ of $W(B_n)$.

Now let $j\in J^{\eta}$ and write $j=\sum_{b,b'\in\mathcal{B}} c_{b,b'}bP(X_1)b'$ for some coefficients $c_{b,b'}\in K$. The condition $\eta(j)=j$ implies that
\[2\Bigl(\sum_{\substack{b\in\mathcal{B}^+\\ b'\in\mathcal{B}^-}}c_{b,b'}bP(X_1)b'+\sum_{\substack{b\in\mathcal{B}^-\\ b'\in\mathcal{B}^+}}c_{b,b'}bP(X_1)b'\Bigr)=0\ ,\]
and since the characteristic of $K$ is different from 2, we obtain
\[j=\sum_{b,b'\in\mathcal{B}^+}c_{b,b'}bP(X_1)b'+\sum_{b,b'\in\mathcal{B}^-}c_{b,b'}bP(X_1)b'\ .\]
The first sum is already in $J'$ since $b,b'\in \mathcal{B}^+\subset \hH(B_n)^{\eta}$. So for $b,b'\in\mathcal{B}^-$ it remains to show that $bP(X_1)b'\in J'$. To do this we write
\[bP(X_1)b'=bg_0^2P(X_1)b'=bg_0P(X_1^{-1})g_0b'\ .\]
We note that $bg_0$ and $g_0b'$ are in $\hH(B_n)^{\eta}$ and, using the assumption on the multiplicity map $m$, 
\[P(X_1^{-1})=\prod_{i\in S_{\bl}} (X_1^{-1}-i)^{m(i)}=\prod_{i\in S_{\bl}}(-iX_1^{-1})^{m(i)} (X_1-i^{-1})^{m(i)}=\prod_{i\in S_{\bl}}(-iX_1^{-1})^{m(i)}\cdot P(X_1)\ .\]
This shows that $P(X_1^{-1})\in J'$ since it is of the form $xP(X_1)$ with $x\in \hH(B_n)^{\eta}$. Therefore we have $bP(X_1)b'\in J'$ and this concludes the proof.
\end{proof}

\paragraph{\textbf{Cyclotomic quotients $W^{\beta}_{\bl,m}$ as subalgebras of fixed points of $V^{\beta}_{\bl,m}$.}} Let $\beta$ be a $W(B_n)$-orbit in $(S_{\bl})^n$, so that all algebras $V^{\beta}_{\bl}, V^{\beta}_{\bl,m}, W^{\beta}_{\bl},W^{\beta}_{\bl,m}$ are defined.

Recall the involutive automorphism $\rho$ of $V^{\beta}_{\bl}$ defined by:
\begin{equation}\label{rho}
\rho\ :\ \ \ \psi_0\mapsto-\psi_0\,,\ \ \ \psi_b\mapsto\psi_b\,,\ b=1,\dots,n-1,\ \ \ \ y_j\mapsto y_j\,,\ j=1,\dots,n,\ \ \ \ \ e(\ti)\mapsto e(\ti)\,,\ \ti\in\beta\ .
\end{equation}
Let $J \subset V^{\beta}_{\bl}$ be the two-sided ideal generated by the elements
\begin{equation}\label{ele-quot2}
P_{\ti}(y_1)=y_1^{m(i_1)}\idmep \ \ \ \textrm{ for } \textbf{i}=(i_1,\dots,i_n) \in \beta\ 
\end{equation}
so that $V^{\beta}_{\bl,m}= V^{\beta}_{\bl}/J$ by definition. We clearly have that $\rho(J)\subset J$ since every element in (\ref{ele-quot2}) is fixed by $\rho$.

Note that, slightly abusing notation, we use the same names for elements of the quotient $V^{\beta}_{\bl,m}$ as for elements of $V^{\beta}_{\bl}$. With this notation, the involutive automorphism $\tilde{\rho}$ of $V^{\beta}_{\bl,m}$ appearing in the conclusion of Lemma \ref{general-lemma} is simply given by the same formulas:
\begin{equation}\label{tilde-rho}
\tilde{\rho}\ :\ \ \ \psi_0\mapsto-\psi_0\,,\ \ \ \psi_b\mapsto\psi_b\,,\ b=1,\dots,n-1,\ \ \ \ y_j\mapsto y_j\,,\ j=1,\dots,n,\ \ \ \ \ e(\ti)\mapsto e(\ti)\,,\ \ti\in\beta\ .
\end{equation}

\begin{proposition}\label{prop-fixed-points2}
Assume that $m$ satisfies $m(i^{-1})=m(i)$ for any $i\in S_{\bl}$. Then the algebra $W^{\beta}_{\bl,m}$ is isomorphic to $\bigl(V_{\bl,m}^{\beta}\bigr)^{\tilde{\rho}}$, with $\tilde{\rho}$ given in (\ref{tilde-rho}).
\end{proposition}
\begin{proof} The proof is similar to the proof of Proposition \ref{prop-fixed-points1}. First, we have by Lemma \ref{general-lemma} that $\bigl(V_{\bl,m}^{\beta}\bigr)^{\tilde{\rho}}$ is isomorphic to $\bigl(V_{\bl}^{\beta}\bigr)^{\rho}/J^{\rho}$.

Besides, by Theorem \ref{prop:fixed point VV is SVV}, we can identify $W^{\beta}_{\bl}$ with the subalgebra of fixed points $\bigl(V_{\bl}^{\beta}\bigr)^{\rho}$. Under this identification, the element $y_1$ in $W^{\beta}_{\bl}$ is sent to the element $y_1$ in $V_{\bl}^{\beta}$. Thus from Definition \ref{def-cyc-quivD}, we have that $W^{\beta}_{\bl,m}$ is the quotient of $\bigl(V_{\bl}^{\beta}\bigr)^{\rho}$ by the ideal generated by the elements $P_{\ti}(y_1)$ in (\ref{ele-quot2}). Denote by $J'$ this ideal of $\bigl(V_{\bl}^{\beta}\bigr)^{\rho}$. We must then prove that $J'=J^{\rho}$.

For the inclusion $J'\subset J^{\rho}$, an element of $J'$ is a linear combination of elements of the form $j=xP_{\ti}(y_1)y$ with $\ti\in\beta$ and $x,y\in \bigl(V_{\bl}^{\beta}\bigr)^{\rho}$. Such an element $j$ obviously lies in $J$ and moreover $\rho(j)=j$ since $P_{\ti}(y_1)$ is invariant par $\rho$ and $x,y\in \bigl(V_{\bl}^{\beta}\bigr)^{\rho}$.

For the reverse inclusion, let $\mathcal{B}$ be the basis of $V_{\bl}^{\beta}$ in Theorem \ref{basis-theorem1}, namely consisting of elements $y_1^{m_1}\cdots y_n^{m_n}\psi_w\idmep$, where  $w \in W(B_n)$, $(m_1,\ldots,m_n) \in (\mathbb{Z}_{\geq 0})^n$ and $\textbf{i}\in \beta$. The set $\mathcal{B}$ splits into two disjoint subsets 
$$\mathcal{B}=\mathcal{B}^+\cup\mathcal{B}^-\,,\ \ \ \ \ \text{where $\mathcal{B}^{\pm}=\{b\in\mathcal{B}\ |\ \rho(b)=\pm b\}$.}$$
The subset $\mathcal{B}^+$ is formed by the basis elements where $w$ is in the subgroup $W(D_n)$ of $W(B_n)$.

Now let $j\in J^{\rho}$. Write $j=\sum_{\ti\in\beta}\sum_{b,b'\in\mathcal{B}} c_{\ti,b,b'}bP_{\ti}(y_1)b'$ for some coefficients $c_{\ti,b,b'}\in K$. The condition $\rho(j)=j$ implies that
\[2\sum_{\ti\in\beta}\Bigl(\sum_{\substack{b\in\mathcal{B}^+\\ b'\in\mathcal{B}^-}}c_{\ti,b,b'}bP_{\ti}(y_1)b'+\sum_{\substack{b\in\mathcal{B}^-\\ b'\in\mathcal{B}^+}}c_{\ti,b,b'}bP_{\ti}(y_1)b'\Bigr)=0\ ,\]
and since the characteristic of $K$ is different from 2, we obtain
\[j=\sum_{\ti\in\beta}\Bigl(\sum_{b,b'\in\mathcal{B}^+}c_{\ti,b,b'}bP_{\ti}(y_1)b'+\sum_{b,b'\in\mathcal{B}^-}c_{\ti,b,b'}bP_{\ti}(y_1)b'\Bigr)\ .\]
The first sum in the parenthesis is already in $J'$ since $b,b'\in \mathcal{B}^+\subset \bigl(V_{\bl}^{\beta}\bigr)^{\rho}$. So for $b,b'\in\mathcal{B}^-$ it remains to show that $bP_{\ti}(y_1)b'\in J'$. To do this we write
\[bP_{\ti}(y_1)b'=b\,y_1^{m(i_1)}e(\ti)\,b'=b\,\psi_0^2y_1^{m(i_1)}e(\ti)\,b'=b\psi_0\,(-y_1)^{m(i_1)}e\bigl(r_0(\ti)\bigr)\,\psi_0b'\ .\]
We note that $b\psi_0$ and $\psi_0b'$ are in $\bigl(V_{\bl}^{\beta}\bigr)^{\rho}$ and, using the assumption on the multiplicity map $m$, 
\[(-y_1)^{m(i_1)}e\bigl(r_0(\ti)\bigr)=(-1)^{m(i_1)}y_1^{m(i^{-1}_1)}e\bigl(r_0(\ti)\bigr)=(-1)^{m(i_1)}P_{r_0(\ti)}(y_1)\,\ .\]
This shows that $(-y_1)^{m(i_1)}e\bigl(r_0(\ti)\bigr)\in J'$. Therefore we have $bP_{\ti}(y_1)b'\in J'$ and this concludes the proof.
\end{proof}

\begin{remark}\label{rem-mult2}
It was noted in Remark \ref{rem-mult} that any cyclotomic quotient $H(B_n)_{\bl,m}$ or $V_{\bl,m}^{\beta}$ is in fact isomorphic to a cyclotomic quotient with the multiplicity map satisfying $m(i^{-1})=m(i)$. This does not have to hold for the cyclotomic quotients $H(D_n)_{\bl,m}$ or $W_{\bl,m}^{\beta}$, and this explains the presence of the assumption in Propositions \ref{prop-fixed-points1} and \ref{prop-fixed-points2}.
\hfill$\triangle$
\end{remark}

\paragraph{Idempotents and blocks of $H(D_n)_{\bl, m}$.} Following the same discussion as in the beginning of Section \ref{sec-cyc} for $H(B_n)_{\bl, m}$, we consider $H(D_n)_{\bl, m}$ as a finite-dimensional representation (by left multiplication) of the subalgebra generated by the commuting elements $X_1, \ldots, X_n$. From \cite[Proposition 3.1]{PW}, we have that all the eigenvalues of $X_1, \ldots, X_n$ belong to $(S_{\bl})^n$, so we can decompose this representation into a direct sum of common generalised eigenspaces for the $X_i$ as follows
\begin{equation*}
H(D_n)_{\bl, m}=\bigoplus_{\ti\in (S_{\bl})^n} M_{\ti}
\end{equation*}
where $M_{\ti}:=\{x\in H(D_n)_{\bl, m}\,\vert\ (X_k-i_k)^Nx=0\,,\ k=1,\dots,n\,,\ \text{for some $N>0$}\}$. We let 
\begin{equation*}
\{ e_{\ti}^H \}_{\ti \in (S_{\bl})^n}
\end{equation*}
denote the associated set of mutually orthogonal idempotents which, by definition, belong to the commutative subalgebra generated by $X_1, \ldots, X_n$. We consider the action of $W(D_n)$ on elements $\ti \in (S_{\bl})^n$ given in Section \ref{sec-defW}. Let $\delta$ denote an orbit in $(S_{\bl})^n$ with respect to this action and write
\begin{equation*}
e_\delta:=\sum_{\ti \in \delta}e^H(\ti).
\end{equation*}
The element $e_\delta$ is therefore a central idempotent in $H(D_n)_{\bl, m}$. Therefore the set $e_\delta H(D_n)_{\bl, m}$ either $\{0\}$ or an idempotent subalgebra (with unit $e_\delta$) which is a union of blocks of $H(D_n)_{\bl, m}$ and we have, similarly to (\ref{direct-sums2}),
\[H(D_n)_{\bl, m}=\bigoplus_{\delta\in\mathcal{O}'}e_\delta H(D_n)_{\bl, m}\ ,\]
$\mathcal{O}'$ is the set of $W(D_n)$-orbits in $(S_{\bl})^n$.

Note that in the theorem below we do not need the assumption on the multiplicity map present in Propositions \ref{prop-fixed-points1} and \ref{prop-fixed-points2}.
\begin{theorem}\label{theorem:main theorem}
The algebras $H(D_n)_{\bl,m}$ and $W_{\bl,m}^{n}$ are isomorphic.
\end{theorem}
\begin{proof}
We claim that we have the following commutative diagram:
\begin{center}
\begin{tikzpicture}
\matrix (m) [matrix of math nodes, row sep=4em,
column sep=5em, text height=1.5ex, text depth=0.25ex]
{H(B_n)_{\bl, m} & H(B_n)_{\bl, m} \\
V^n_{\bl, m} & V^n_{\bl, m} \\};
\path[-stealth]
(m-1-1) edge node [above] {$\tilde{\eta}$} (m-1-2)
(m-2-1) edge node [left] {$\sim$} (m-1-1)
(m-1-2) edge node [right] {$\sim$} (m-2-2)
(m-2-1) edge node [below] {$\tilde{\rho}$} (m-2-2);
\end{tikzpicture}
\end{center}
The vertical maps are the isomorphisms from Theorem \ref{theorem:type B iso when p^2=1} and the horizontal maps are the involutive automorphisms considered just above. The commutativity of the diagram follows directly from the explicit description of $\tilde{\eta}$ in (\ref{tilde-eta}) and $\tilde{\rho}$ in (\ref{tilde-rho}), together with the fact that the isomorphisms in Theorem \ref{theorem:type B iso when p^2=1} send $g_0$ to $\psi_0$ (and vice-versa) while neither $g_0$ nor $\psi_0$ appear in the images of the other generators; see the formulas in (\ref{iso-maps}).

From the commutative diagram we immediately deduce the algebra isomorphism
\begin{equation*}
(H(B_n)_{\bl,m})^{\tilde{\eta}} \cong \bigl(V_{\bl, m}^{n}\bigr)^{\tilde{\rho}}.
\end{equation*}

Now assume that the multiplicity map satisfies $m(i^{-1})=m(i)$ for any $i\in S_{\bl}$. Then, by Proposition \ref{prop-fixed-points1} the left-hand side is isomorphic to $H(D_n)_{\bl, m}$ and, by Proposition \ref{prop-fixed-points2}, taking direct sums over $\beta\in\mathcal{O}$, the right-hand side is isomorphic to $W^n_{\bl,m}$ thus proving the result.

Now consider the general situation with no assumption on $m$. We define a new multiplicity map $\tilde{m}$ by $\tilde{m}(i)=\text{max}\{m(i),m(i^{-1})\}$ for all $i\in S_{\bl}$. Then this multiplicity map satisfies $\tilde{m}(i^{-1})=\tilde{m}(i)$ for all $i\in S_{\bl}$, and the first part of the proof shows that
\begin{equation}\label{iso-interD}
H(D_n)_{\bl,\tilde{m}} \cong W_{\bl, \tilde{m}}^{n}\ .
\end{equation}
Then in $H(D_n)_{\bl,\tilde{m}}$, denote $I_H$ the ideal generated by the element
\[P(X_1)=\prod_{i\in S_{\bl}} (X_1-i)^{m(i)}\ ,\]
and in $W_{\bl, \tilde{m}}^{n}$, denote $I_W$ the ideal generated by the elements
\[P_{\ti}(y_1)=y_1^{m(i_1)}\idmep \ \ \ \textrm{ for } \textbf{i} \in (S_{\bl})^n\ .\]
Since $\tilde{m}(i)\geq m(i)$ for all $i\in S_{\bl}$, we have from the definitions that
\[H(D_n)_{\bl,m}=H(D_n)_{\bl,\tilde{m}}/I_H\ \ \ \ \ \text{and}\ \ \ \ \ \ W_{\bl, m}^{n}=W_{\bl, \tilde{m}}^{n}/I_W\ .\]
So it remains to show that $I_H$ is sent to $I_W$ by the isomorphism in (\ref{iso-interD}). Denoting this isomorphism by $\Theta$, we recall from the proof of Theorem \ref{theorem:type B iso when p^2=1} (namely the formulas in (\ref{iso-maps})) that we have:
\[\Theta(X_1)=\sum_{\ti\in (S_{\bl})^n}i_1(1-g(y_1))e(\ti)\ \ \ \ \ \ \ \text{and}\ \ \ \ \ \ \ \Theta^{-1}(y_1)=\sum_{\ti\in (S_{\bl})^n}f(1-i_1^{-1}X_1)e^{H}(\ti)\ ,\]
where $f$ is the power series $z+\frac{z}{1-z}$ and $g$ is its composition inverse. Both $f$ and $g$ have a constant term equal to 0.

For the inclusion $\Theta(I_H)\subset I_W$, we have, using that $e(\ti)$ are orthogonal idempotents,
\[\begin{array}{ll}
\Theta(P(X_1)) & =\displaystyle\prod_{i\in S_{\bl}}\Bigl(\sum_{\ti\in (S_{\bl})^n}(i_1-i_1g(y_1)-i)e(\ti)\Bigr)^{m(i)}\\[1.4em]
 & = \displaystyle\sum_{\ti\in (S_{\bl})^n}\prod_{i\in S_{\bl}}(i_1-i_1g(y_1)-i)^{m(i)}e(\ti)\ .
 \end{array}\]
For each $\ti\in (S_{\bl})^n$, pick the factor corresponding to $i=i_1$ in the product. Since $g$ has no constant term, this factor is proportional to $y_1^{m(i_1)}e(\ti)$ and therefore in $I_W$. We conclude that $\Theta(P(X_1))\in I_W$.

For the inclusion $\Theta^{-1}(I_W)\subset I_H$, let $\ti\in (S_{\bl})^n$. We have, using that $e^H(\textbf{j})$ are orthogonal idempotents,
\[\begin{array}{ll}
\Theta^{-1}(P_{\ti}(y_1)) & =\displaystyle\Bigl(\sum_{\textbf{j}\in (S_{\bl})^n}f(1-j_1^{-1}X_1)e^{H}(\textbf{j})\Bigr)^{m(i_1)}e^H(\ti)\\
 & = \displaystyle f(1-i_1^{-1}X_1)^{m(i_1)}e^H(\ti)\ .
 \end{array}\]
Since $f$ has no constant term, this is proportional to $(X_1-i_1)^{m(i_1)}e^H(\ti)$. Moreover, for any $i\neq i_1$, by construction of the idempotent $e^H(\ti)$, the element $(X_1-i)e^H(\ti)$ is invertible in $K[X_1]e^H(\ti)$. In other words, there exists a polynomial $R(X_1)$ in $X_1$ such that $R(X_1)(X_1-i)e^H(\ti)=e^H(\ti)$. For convenience, we denote here $R(X_1)$ by $(X_1-i)^{-1}$. With this notation, we have:
\[(X_1-i_1)^{m(i_1)}e^H(\ti)=\prod_{\substack{i\in S_{\bl}\\ i\neq i_1}}(X_1-i)^{-m(i)}\cdot P(X_1)e^H(\ti)\ .\]
Thus we see that $(X_1-i_1)^{m(i_1)}e^H(\ti)$ is proportional in $H(D_n)_{\bl,\tilde{m}}$ to the element $P(X_1)$. We conclude that $(X_1-i_1)^{m(i_1)}e^H(\ti)$ and in turn $\Theta^{-1}(P_{\ti}(y_1))$ is in the ideal $I_H$.

We conclude that $\Theta(I_H)=I_W$ and the proof of the theorem is complete.
\end{proof}

In fact, as in Theorem \ref{theorem:type B iso when p^2=1}, a more precise result is true describing each idempotent subalgebra $e_{\delta}H(D_n)_{\bl,m}$.

\begin{corollary}
The algebras $e_{\delta}H(D_n)_{\bl,m}$ and $W_{\bl,m}^{\delta}$ are isomorphic for any orbit $\delta\in\mathcal{O}'$.
\end{corollary}
\begin{proof}
Let $\delta\in\mathcal{O}'$. The central idempotent in $W_{\bl,m}^{n}$ projecting onto $W_{\bl,m}^{\delta}$ in the direct sum (\ref{direct-sums2}) is $\sum_{\ti\in\delta}e(\ti)$. It is immediate to see (one has to look at the maps in (\ref{iso-maps})) that under the isomorphism of Theorem \ref{theorem:main theorem}, this element is sent to $e_{\delta}$ in $H(D_n)_{\bl,m}$. The corollary then follows directly.
\end{proof}

\noindent L. P.d'A.: { \sl \small Universit\'e de Reims Champagne Ardenne, CNRS, LMR FRE 2011, 51097, Reims, France} 
\newline \noindent {\tt \small email: loic.poulain-dandecy@univ-reims.fr}\\

\noindent R. W.: { \sl \small Universit\'{e} Paris Diderot-Paris VII, B\^{a}timent Sophie Germain, 75205 Paris Cedex 13 France} 
\newline \noindent {\tt \small email: ruari.walker@imj-prg.fr}

\end{document}